\documentclass[12pt]{amsart}

\usepackage{amssymb}
\usepackage[cmtip,all]{xy}
\usepackage{verbatim}
\usepackage{upref}
\usepackage[usenames,dvipsnames]{color}
\usepackage{url}
\usepackage[normalem]{ulem}

\newtheorem{thm}{Theorem}[section]
\newtheorem*{thm*}{Theorem}
\newtheorem{lem}[thm]{Lemma}
\newtheorem{cor}[thm]{Corollary}
\newtheorem{prop}[thm]{Proposition}

\theoremstyle{definition} 
\newtheorem{defn}[thm]{Definition}

\newtheorem{ex}[thm]{Example}

\theoremstyle{remark} 
\newtheorem{rem}[thm]{Remark}

\numberwithin{equation}{section}

\newcommand{\secref}[1]{Section~\textup{\ref{#1}}}

\newcommand{\thmref}[1]{Theorem~\textup{\ref{#1}}}

\newcommand{\lemref}[1]{Lemma~\textup{\ref{#1}}}
\newcommand{\propref}[1]{Proposition~\textup{\ref{#1}}}

\newcommand{\defnref}[1]{Definition~\textup{\ref{#1}}}

\newcommand{\exref}[1]{Example~\textup{\ref{#1}}}
\newcommand{\remref}[1]{Remark~\textup{\ref{#1}}}

\newcommand{\KK}{\mathcal K}
\newcommand{\FF}{\mathcal F}
\newcommand{\LL}{\mathcal L}
\newcommand{\RR}{\mathcal R}
\newcommand{\TT}{\mathcal T}

\renewcommand{\SS}{\mathcal S}

\newcommand{\C}{\mathbb C}

\newcommand{\variso}{\overset{\simeq}{\longrightarrow}}

\renewcommand{\bar}{\overline}
\newcommand{\what}{\widehat}
\newcommand{\wilde}{\widetilde}
\newcommand{\inv}{^{-1}}
\newcommand{\<}{\langle}
\renewcommand{\>}{\rangle}
\newcommand{\smtx}[1]{\left(\begin{smallmatrix} #1
\end{smallmatrix}\right)}

\newcommand{\ann}{^\perp}
\newcommand{\pann}{{}\ann}
\renewcommand{\:}{\colon}
\newcommand{\cotimes}{\mathrel{\sharp}}

\renewcommand{\)}{\textup)}
\newcommand{\rt}{\textup{rt}}

\newcommand{\id}{\text{\textup{id}}}

\DeclareMathOperator{\ad}{Ad}
\DeclareMathOperator*{\spn}{span}
\DeclareMathOperator*{\clspn}{\overline{\spn}}
\DeclareMathOperator{\ind}{Ind}
\DeclareMathOperator{\dashind}{\!-Ind}
\DeclareMathOperator{\glb}{glb}

\newcommand{\midtext}[1]{\quad\text{#1}\quad}
\newcommand{\righttext}[1]{\quad\text{#1 }}

\DeclareMathOperator{\cp}{CP}

\begin{document}
\title{Coaction functors, II}
\author[Kaliszewski]{S. Kaliszewski}
\address{School of Mathematical and Statistical Sciences
\\Arizona State University
\\Tempe, Arizona 85287}
\email{kaliszewski@asu.edu}
\author[Landstad]{Magnus~B. Landstad}
\address{Department of Mathematical Sciences\\
Norwegian University of Science and Technology\\
NO-7491 Trondheim, Norway}
\email{magnusla@math.ntnu.no}
\author[Quigg]{John Quigg}
\address{School of Mathematical and Statistical Sciences
\\Arizona State University
\\Tempe, Arizona 85287}
\email{quigg@asu.edu}

\date{\today}


\begin{abstract}
In further study of the application of crossed-product functors to the Baum-Connes Conjecture,
Buss, Echterhoff, and Willett
introduced various other properties that crossed-product functors may have.
Here we introduce and study analogues of these properties for coaction functors, making sure that the properties are preserved when the coaction functors are composed with the full crossed product to make a crossed-product functor.
The new properties for coaction functors studied here are
functoriality for generalized homomorphisms
and the correspondence property.
We particularly study the connections with the ideal property.
The study of functoriality for generalized homomorphisms
requires a detailed development of the Fischer construction of maximalization of coactions with regard to possibly degenerate homomorphisms into multiplier algebras.
We verify that all ``KLQ'' functors arising from large ideals of the Fourier-Stieltjes algebra $B(G)$ have all the properties we study,
and at the opposite extreme we give an example of a coaction functor having none of the properties.
\end{abstract}
\maketitle

\section{Introduction}\label{intro}

As part of their study of the Baum-Connes Conjecture, \cite{bgwexact} considered \emph{exotic crossed products} between the full and reduced crossed products of a $C^*$-dynamical system,
and a crucial feature was that the construction be \emph{functorial}
for equivariant homomorphisms.
In \cite{klqfunctor} we introduced a two-step construction of crossed-product functors:
first form the full crossed product, then apply a \emph{coaction functor}.
Although this recipe does not give all crossed-product functors, there is some evidence that it might produce the functors that are most important for the program of \cite{bgwexact}.

In \cite{bgwexact}, the applications to the Baum-Connes Conjecture lead to the desire that the crossed-product functors be \emph{exact} and \emph{Morita compatible}, and it was proved that there is a smallest (for a suitable partial ordering) crossed product with these properties.
The idea is that every family of crossed-product functors has a greatest lower bound, and that exactness and Morita compatibility are preserved by greatest lower bounds.
In \cite{klqfunctor} we proved analogues of these facts for coaction functors.

In further study of the application of crossed-product functors to the Baum-Connes Conjecture,
\cite{bew} studied various other properties that crossed-product functors may have.
This motivated us to investigate in the current paper the analogous properties of coaction functors.

There is a subtlety regarding the appropriate choices of categories.
To study short exact sequences, the morphisms should be homomorphisms between the $C^*$-algebras themselves, and we call the resulting categories \emph{classical}.
On the other hand, some of the properties considered in \cite{bew} require homomorphisms into multiplier algebras.
Most of the literature on noncommutative $C^*$-crossed-product duality
uses \emph{nondegenerate categories},
where the morphisms are
nondegenerate homomorphisms into multiplier algebras;
the nondegeneracy guarantees that the maps can be composed.
On the other hand,
for some of the properties studied in \cite{bew} it is actually important to allow \emph{possibly degenerate} homomorphisms into multiplier algebras.
Of course this is problematic in terms of composing morphisms,
but nevertheless \cite{bew} introduced a reasonable notation of
\emph{functoriality for generalized homomorphisms},
involving such possibly degenerate homomorphisms.
In this paper we chose to develop the theory along three parallel tracks:
first we prove what we can in the context of generalized homomorphisms,
then we specialize to the classical and the nondegenerate categories.
However, our main interest is in the classical categories,
and for much of this paper the classical case will be our default, with occasional mention of nondegenerate categories.

Nondegenerate equivariant categories have been fairly well-studied,
but (perhaps unexpectedly) the classical counterparts have not,
especially in noncommutative crossed-product duality.
In \cite{klqfunctor} we began to fill in some of these gaps in the theory of classical categories,
and here we will continue this,
to prepare the way for our study of analogues for coaction functors of some of the properties introduced in \cite{bew}.
In \cite{klqfunctor} we gave a brief indication of how maximalization of coactions is a functor on the classical category of coactions, which we make more precise in \secref{max gen sec}.

We begin in \secref{prelim} by recording a few of our conventions for coactions and actions.
We also discuss the distinction between \emph{nondegenerate} and \emph{classical} categories of $C^*$-algebras with extra structure. For the study of exactness of coaction functors, the classical categories are appropriate, so we focus upon them in this paper.
Coaction functors involve maximalization of coactions,
and we outline Fischer's construction of maximalization as a composition of three simpler functors.
We finish \secref{prelim} with a short discussion of coaction functors, taken from \cite{klqfunctor} and \cite{klqexact}.
In particular, we recall a few properties that coaction functors may have:
\emph{exactness},
\emph{Morita compatibility},
and
the \emph{ideal property}.
The first of these occupies a central position in the application of coaction functors to the crossed-product functors of \cite{bgwexact},
while the second and third are analogues of properties of action-crossed-product functors discussed in \cite{bew}.
In \propref{glb exact or morita} we record a more precise statement of a result in \cite{klqfunctor} regarding greatest lower bounds of exact or Morita compatible coaction functors.
The whole point of coaction functors is that they give a large
(albeit not exhaustive)
source of crossed-product functors in the sense of \cite{bgwexact}.
There are numerous open problems regarding the relationship between these two types of functors,
and in \secref{prelim} we mention one of these, involving greatest lower bounds.
We also recall another type of coaction functor:
\emph{decreasing},
which include those coaction functors arising from \emph{large ideals} of the Fourier-Stieltjes algebra $B(G)$;
the associated crossed-product functors for actions have been referred to as
``KLQ functors'' \cite{bew, bew2} or ``KLQ crossed products'' \cite{bgwexact}.

In \secref{max gen sec} we discuss how to maximalize possibly degenerate equivariant homomorphisms into multiplier algebras, with an eye toward developing an analogue for coaction functors of the \emph{functoriality for generalized homomorphisms} discussed in \cite{bew}.
This requires consideration of generalized homomorphisms for each of the three steps in the Fischer construction.
As a side benefit, we close \secref{max gen sec} by remarking how \thmref{max gen} gives a more precise justification
than that one in \cite[Section~3]{klqfunctor}
that maximalization is a functor on the classical category of coactions.

In \secref{gen hom} we introduce an analogue for coaction functors of the property
called \emph{functoriality for generalized homomorphisms}
in \cite{bew}.
Here the term ``generalized homomorphism'' refers to a possibly degenerate homomorphism $\phi\:A\to M(B)$; these are somewhat delicate, and some care must be exercised in dealing with them.
We prove some analogues for coaction functors of results of \cite{bew};
for example, coaction functors that are functorial for generalized homomorphisms
in the sense of \defnref{gen def}
satisfy a limited version of the usual composability aspect of actual functors,
and
every functor arising from a large ideal of $B(G)$ has this generalized functoriality property.
We also give a further discussion of the ideal property,
in particular proving that it is implied by functoriality for generalized homomorphisms.
This is weaker than the corresponding result of \cite{bew},
namely that for crossed-product functors these two properties are equivalent.
We also prove that both the ideal property and functoriality for generalized homomorphisms
are inherited by greatest lower bounds.

In \secref{cor ppy sec} we introduce the \emph{correspondence property} for coaction functors,
which is
an analogue of the \emph{correspondence crossed-product functors} of \cite{bew}.
This is much stronger than Morita compatibility,
and we need to do a bit of work to develop it.
As a side benefit of this work, we prove that if a coaction functor is Morita compatible then the associated crossed-product functor for actions is strongly Morita compatible in the sense of \cite{bew},
and we also prove a technical lemma showing that, in the presence of the ideal property, the test for Morita compatibility can be relaxed somewhat.
We prove that a coaction functor has the correspondence property if and only if it is both Morita compatible and functorial for generalized homomorphisms,
which is an analogue of a similar equivalence for crossed-product functors in \cite{bew}.
It follows that if a coaction functor has the correspondence property then the associated crossed-product functor for actions is a correspondence crossed-product functor in the sense of \cite{bew}.
Among the consequences, we deduce that
every coaction functor arising from a large ideal of $B(G)$ has the correspondence property,
and that the correspondence property is inherited by greatest lower bounds,
so that in particular there is a smallest coaction functor with the correspondence property.
Also, a result of \cite{bew} showing that the output of a correspondence crossed-product functor carries a quotient of the dual coaction on the full crossed product strengthens our belief that the most important crossed-product functors are those arising from coaction functors.

\section{Preliminaries}\label{prelim}

Throughout, $G$ will be a locally compact group,
$A,B,C,D$ will be $C^*$-algebras,
actions of $G$ are denoted by letters such as $\alpha,\beta,\gamma$,
and coactions of $G$ by letters such as $\delta,\epsilon,\zeta$.
Throughout, we assume that $G$ is second countable,
so that the Hilbert space $L^2(G)$ will be separable;
second countability of $G$
is needed for the use of Fischer's result, and in
that proof separability of $L^2(G)$ is essential.
We refer to \cite[Appendix~A]{enchilada} and \cite{maximal} for conventions regarding actions and coactions,
and to \cite[Chapters~1--2]{enchilada} for $C^*$-correspondences\footnote{called \emph{right-Hilbert bimodules} in \cite{enchilada}} and imprimitivity bimodules.

We write $A\rtimes_\alpha G$ for the crossed product of an action $(A,\alpha)$,
and $(i_A,i_G)$ for the universal covariant homomorphism from $(A,G)$ to the multiplier algebra $M(A\rtimes_\alpha G)$,
occasionally writing $i_G^\alpha$ to avoid ambiguity.
We write $\what\alpha$ for the dual coaction.

We write $A\rtimes_\delta G$ for the crossed product of a coaction $(A,\delta)$,
and $(j_A,j_G)$ for the universal covariant homomorphism from $(A,C_0(G))$ to $M(A\rtimes_\delta G)$,
occasionally writing $j_G^\delta$ to avoid ambiguity.
We write $\what\delta$ for the dual action.

Given a coaction $(A,\delta)$, we find it convenient to use the associated $B(G)$-module structure
given by
\[
f\cdot a=(\id\otimes f)\circ\delta(a)\righttext{for}f\in B(G),a\in A,
\]
and in \cite[Appendix~A]{klqfunctor} we recorded a few properties.
We will need the following mild strengthening of \cite[Proposition~A.1]{klqfunctor}:

\begin{prop}\label{equivariant}
Let $(A,\delta)$ and $(B,\epsilon)$ be coactions of $G$, and let $\phi:A\to M(B)$ be a homomorphism. Then $\phi$ is $\delta-\epsilon$ equivariant if and only if
it is a module map, i.e., 
\[
\phi(f\cdot a)=f\cdot \phi(a)\righttext{for all}f\in B(G),a\in A.
\]
\end{prop}

\begin{proof}
As we mentioned in \cite[proof of Lemma~3.17]{klqexact}, the argument of \cite[Proposition~A.1]{klqfunctor} carries over, with the minor adjustment that in the second line of the multiline displayed computation 
the map $\phi\otimes\id$ must be replaced by the canonical extension
\[
\bar{\phi\otimes\id}\:\wilde M(A\otimes C^*(G))\to M(B\otimes C^*(G)),
\]
which exists by \cite[Proposition~A.6]{enchilada},
and where we recall the notation
\begin{align*}
&\wilde M(A\otimes C^*(G))
=\{m\in M(A\otimes C^*(G)):
\\&\hspace{.5in} m(1\otimes C^*(G))\cup (1\otimes C^*(G))m\subset A\otimes C^*(G)\}.
\qedhere
\end{align*}
\end{proof}

\subsection*{Classical and nondegenerate categories}

In all of our categories, the objects will be $C^*$-algebras, usually equipped with some extra structure, and the morphisms will be homomorphisms that preserve this extra structure in some sense.
We consider two main types of homomorphisms:
\emph{nondegenerate} homomorphisms $\phi\:A\to M(B)$,
and what we call
\emph{classical} homomorphisms $\phi\:A\to B$,
and these give rise to what we call \emph{nondegenerate} and \emph{classical} categories, respectively.
We are concerned mainly with the classical case,
but occasionally we will refer to the nondegenerate case,
and sometimes we will develop the two in parallel.
We also need to consider what Buss, Echterhoff, and Willett call \emph{generalized homomorphisms} $\phi\:A\to M(B)$, which are allowed to be degenerate.
Perhaps surprisingly, in the noncommutative crossed-product duality literature,
the nondegenerate categories are used almost exclusively;
here we will devote more attention to developing the tools we need for the classical categories.

Warning: in this paper we will slightly modify some of the notation from \cite{klqfunctor}:
given a coaction $(A,\delta)$,
recall from \cite{maximal} that $\delta$ is called \emph{maximal} if the canonical map $\Phi\:A\rtimes_\delta G\rtimes_{\what\delta} G\to A\otimes \KK(L^2(G))$ is an isomorphism,
and that an arbitrary $(A,\delta)$
has a \emph{maximalization},
which is a maximal coaction $(A^m,\delta^m)$ and a $\delta^m-\delta$ equivariant surjection,
which 
we will write
as
$\psi_A\:A^m\to A$, rather than $q^m_A$,
having the property that
$\psi_A\rtimes G\:A^m\rtimes_{\delta^m} G\to A\rtimes_\delta G$
is an isomorphism.
On the nondegenerate category of coactions,
Fischer proves that $\psi$ gives a natural transformation from maximalization to the identity functor;
in \cite{klqfunctor} we stated this for the classical category,
and we will make this more precise in \thmref{max gen}.

On the other hand, we will use the same notation as in \cite{klqfunctor} for
the surjections
$\Lambda_A\:A\to A^n$ giving a natural transformation
from the identity functor to
the normalization functor
$(A,\delta)\mapsto (A^n,\delta^n)$
(for both the classical and the nondegenerate categories).

Given a coaction $(A,\delta)$,
we call 
a $C^*$-subalgebra $B$ of $M(A)$ \emph{strongly $\delta$-invariant} if
\[
\clspn\{\delta(B)(1\otimes C^*(G))\}=B\otimes C^*(G),
\]
in which case
by \cite[Lemma~1.6]{fullred}
$\delta$ restricts to a coaction $\delta_B$ on $B$.
If $I$ is a strongly $\delta$-invariant ideal of $A$, then
by \cite[Propositions~2.1 and 2.2, Theorem~2.3]{nil:full}
(see also \cite[Proposition~4.8]{lprs}),
$I\rtimes_{\delta_I} G$ can be naturally identified with an ideal of $A\rtimes_\delta G$,
and $\delta$ descends to a coaction $\delta^I$ on $A/I$
in such a manner that
\[
0
\to I\rtimes_{\delta_I} G
\to A\rtimes_\delta G
\to (A/I)\rtimes_{\delta^I} G
\to 0
\]
is a short exact sequence in the classical category of coactions.

\begin{rem}
Given a coaction $(A,\delta)$ and an ideal $I$ of $A$,
the existence of a coaction $\delta^I$ on the quotient $A/I$
such that the quotient map $A\to A/I$ is $\delta-\delta^I$ equivariant
is a weaker condition than the above strong invariance,
and when it is satisfied we say that $\delta$ descends to a coaction on $A/I$.
\end{rem}

\subsection*{The Fischer construction}
For convenient reference we record the following rough outline of
Fischer's construction
of the maximalization of a coaction $(A,\delta)$
\cite[Section~6]{fischer}
(see also \cite{koqstable} and \cite{koqmaximal}).
First of all, letting $\KK$ denote the algebra of compact operators on a separable infinite-dimensional Hilbert space, a \emph{$\KK$-algebra}
is a pair $(A,\iota)$,
where $A$ is a $C^*$-algebra
and $\iota\:\KK\to M(A)$ is a nondegenerate homomorphism.
Given a $\KK$-algebra $(A,\iota)$,
the \emph{$A$-relative commutant of $\KK$} is
\[
C(A,\iota):=\{m\in M(A):m\iota(k)=\iota(k)m\in A\text{ for all }k\in\KK\}.
\]
The \emph{canonical isomorphism}
$\theta_A\:C(A,\iota)\otimes\KK\variso A$
is determined by
$\theta_A(a\otimes k)=a\iota(k)$ for $a\in A,k\in\KK$
(see \cite[Remark~3.1]{fischer} and \cite[Proposition~3.4]{koqstable}).
If $(B,\jmath)$ is another $\KK$-algebra and $\phi\:A\to M(B)$ is a nondegenerate homomorphism such that $\phi\circ\iota=\jmath$, then there is a unique nondegenerate homomorphism
$C(\phi)\:C(A,\iota)\to M(C(B,\jmath))$ making the diagram
\[
\xymatrix@C+30pt{
A \ar[r]^-\phi
&M(B)
\\
C(A,\iota)\otimes\KK \ar[u]^{\theta_A} \ar[r]_-{C(\phi)\otimes\id}
&M(C(B,\jmath)\otimes\KK) \ar[u]_{\theta_B}
}
\]
commute.

A \emph{$\KK$-coaction} is a triple $(A,\delta,\iota)$,
where $(A,\delta)$ is a coaction and $(A,\iota)$ is a $\KK$-algebra
such that $\delta\circ\iota=\iota\otimes 1$.
If $(A,\delta,\iota)$ is a $\KK$-coaction, then
the relative commutant $C(A,\iota)$ is strongly $\delta$-invariant,
and the restricted coaction $C(\delta)=\delta|_{C(A,\iota)}$
is maximal if $\delta$ is,
and $\theta_A$ is
$(C(\delta)\otimes_*\id)-\delta$ equivariant \cite[Lemma~3.2]{koqmaximal}.

An \emph{equivariant action} is a triple $(A,\alpha,\mu)$,
where $(A,\alpha)$ is an action of $G$ and $\mu\:C_0(G)\to M(A)$ is a nondegenerate $\rt-\alpha$ equivariant homomorphism,
and
where
in turn
$\rt$ is the action of $G$ on $C_0(G)$ given by $\rt_s(f)(t)=f(ts)$.

A \emph{cocycle} for a coaction $(A,\delta)$ is a unitary element $U\in M(A\otimes C^*(G))$
such that
$(\id\otimes \delta_G)(U)=(U\otimes 1)(\delta\otimes\id)(U)$
and
$\ad U\circ\delta(A)(1\otimes C^*(G))\subset A\otimes C^*(G)$.
Then $\ad U\circ\delta$ is a coaction on $A$,
and is Morita equivalent to $\delta$,
and hence is maximal if and only if $\delta$ is.
If $U$ is a $\delta$-cocycle,
$(B,\epsilon)$ is another coaction,
and $\phi\:A\to M(B)$ is a nondegenerate
$\delta-\epsilon$ equivariant homomorphism,
then
$(\phi\otimes\id)(U)$ is an $\epsilon$-cocycle and
$\phi$ is $\ad U\circ\delta-\ad(\phi\otimes\id)(U)\circ\epsilon$ equivariant.

Given an equivariant action $(A,\alpha,\mu)$,
the unitary element
\[
V_A:=((i_A\circ\mu)\otimes\id)(w_G)
\]
is an $\what\alpha$-cocycle,
and we write $\wilde\alpha=\ad V_A\circ\what\alpha$.
Then
$(A\rtimes_\alpha G,\wilde\alpha,\mu\rtimes G)$ is a maximal $\KK$-coaction
\cite[Lemma~3.1]{koqmaximal}.

Now, if $(A,\delta)$ is a coaction,
then $(A\rtimes_\delta G,\what\delta,j_G)$ is an equivariant action,
so
$(A\rtimes_\delta G\rtimes_{\what\delta} G,\wilde{\what\delta},j_G\rtimes G)$
is a $\KK$-coaction,
and hence
\[
(A^m,\delta^m):=
\bigl(C(A\rtimes_\delta G\rtimes_{\what\delta} G,j_G\rtimes G),C(\wilde{\what\delta})\bigr)
\]
is a maximal coaction.
Letting
\[
\Phi_A\:A\rtimes_\delta G\rtimes_{\what\delta} G\to A\otimes\KK
\]
be the \emph{canonical surjection}, which is
$\wilde{\what\delta}-(\delta\otimes_*\id)$
equivariant,
Fischer proves that there is a unique $\delta^m-\delta$ equivariant surjective homomorphism $\psi_A\:A^m\to A$ such that the diagram
\[
\xymatrix{
&A\rtimes_\delta G\times_{\what\delta} G \ar[dr]^{\Phi_A}
\\
A^m\otimes\KK \ar[ur]^{\theta_{A\rtimes_\delta G\rtimes_{\what\delta} G}}
\ar[rr]_{\psi_A\otimes\id}
&&A\otimes\KK
}
\]
commutes,
and moreover $\psi_A\:(A^m,\delta^m)\to (A,\delta)$ is a maximalization of $(A,\delta)$.
Fischer goes on to prove that maximalization is a functor on the nondegenerate category of coactions, by showing that if $\phi\:A\to M(B)$ is a nondegenerate $\delta-\epsilon$ equivariant homomorphism
then there is a unique homomorphism
$\phi^m\:A^m\to M(B^m)$
making the diagram
\begin{equation*}
\xymatrix{
&A\rtimes_\delta G\rtimes_{\what\delta} G \ar[dr]^{\Phi_A} \ar'[d][dd]^{\phi\rtimes G\rtimes G}
\\
A^m\otimes\KK \ar[ur]^(.4){\theta_{A\rtimes_\delta G\rtimes_{\what\delta} G}}_\simeq
\ar[rr]_(.3){\psi_A\otimes\id} \ar[dd]_{\phi^m\otimes\id}
&&A\otimes\KK \ar[dd]^{\phi\otimes\id}
\\
&M(B\rtimes_\epsilon G\rtimes_{\what\epsilon} G) \ar[dr]^{\Phi_B}
\\
M(B^m\otimes\KK) \ar[ur]^\simeq_(.6){\theta_{B\rtimes_\epsilon G\rtimes_{\what\epsilon} G}}
\ar[rr]_-{\psi_B\otimes\id}
&&M(B\otimes\KK)
}
\end{equation*}
commute.
Consequently,
the diagram
\begin{equation*}
\xymatrix@C+20pt{
A^m \ar[r]^-{\phi^m} \ar[d]_{\psi_A}
&M(B^m) \ar[d]^{\psi_B}
\\
A \ar[r]_-\phi
&M(B)
}
\end{equation*}
also commutes, and
$\phi^m$
is 
nondegenerate and
$\delta^m-\epsilon^m$ equivariant.

\subsection*{Coaction functors}
A functor $\tau\:(A,\delta)\mapsto (A^\tau,\delta^\tau)$,
$\phi\mapsto \phi^\tau$
on the classical category of coactions is a \emph{coaction functor} if
it fits into a commutative diagram
\begin{equation}\label{co fn diag}
\xymatrix{
&(A^m,\delta^m) \ar[dl]_{\psi_A} \ar[dr]^{q^\tau_A}
\\
(A,\delta) \ar[dr]_{\Lambda_A}
&&
(A^\tau,\delta^\tau) \ar[dl]^{\Lambda^\tau_A}
\\
&(A^n,\delta^n)
}
\end{equation}
of surjective natural transformations.
In \cite[Lemma~4.3]{klqfunctor} we proved that the existence of the natural transformation $\Lambda^\tau$ is automatic, provided we insist that $\ker q^\tau_A\subset \ker \Lambda_A\circ\psi_A$.

We observed in \cite[Example~4.2]{klqfunctor} that
maximalization, normalization, and the identity functor are all coaction functors.

Given two coaction functors $\tau$ and $\sigma$,
we say $\sigma$ is \emph{smaller} than $\tau$, written $\sigma\le \tau$,
if there is a natural transformation $\Gamma^{\tau,\sigma}$ fitting into commutative diagrams
\[
\xymatrix{
&(A^m,\delta^m) \ar[dl]_{q_A^\tau} \ar[dr]^{q^\sigma_A}
\\
(A^\tau,\delta^\tau) \ar[dr]_{\Lambda^\tau_A} \ar[rr]^{\Gamma^{\tau,\sigma}_A}
&&
(A^\sigma,\delta^\sigma) \ar[dl]^{\Lambda^\sigma_A}
\\
&(A^n,\delta^n),
}
\]
in other words, $\ker q^\tau_A\subset \ker q^\sigma_A$.
In \cite[Theorem~4.9]{klqfunctor} we proved that every nonempty family $\TT$ of coaction functors has a greatest lower bound $\glb \TT$, characterized by
\[
\ker q^{\glb\TT}=\clspn_{\tau\in \TT}\ker q^\tau.
\]

A coaction functor $\tau$ is \emph{exact} \cite[Definition~4.10]{klqfunctor}
if for every short exact sequence
\[
\xymatrix{
0 \ar[r]
&(I,\gamma) \ar[r]^-\phi
&(A,\delta) \ar[r]^-\psi
&(B,\epsilon) \ar[r]
&0
}
\]
in the classical category of coactions the image
\[
\xymatrix{
0 \ar[r]
&(I^\tau,\gamma^\tau) \ar[r]^-{\phi^\tau}
&(A^\tau,\delta^\tau) \ar[r]^-{\psi^\tau}
&(B^\tau,\epsilon^\tau) \ar[r]
&0
}
\]
under $\tau$ is also exact.
Maximalization is exact \cite[Theorem~4.11]{klqfunctor}.

A coaction functor $\tau$ is \emph{Morita compatible} \cite[Definition~4.16]{klqfunctor}
if for every $(A,\delta)-(B,\epsilon)$ imprimitivity-bimodule coaction $(X,\zeta)$,
with associated $(A^m,\delta^m)-(B^m,\epsilon^m)$ imprimitivity-bimodule coaction $(X^m,\zeta^m)$,
the Rieffel correspondence of ideals satisfies
\[
\ker q^\tau_A=X^m\dashind \ker q^\tau_B,
\]
equivalently there are an $A^\tau-B^\tau$ imprimitivity bimodule $X^\tau$ and
a surjective $q^\tau_A-q^\tau_B$ compatible imprimitivity-bimodule homomorphism
$q^\tau_X\:X^m\to X^\tau$ \cite[Lemma~4.19]{klqfunctor}.
Trivially, maximalization is Morita compatible,
and routine linking-algebra techniques show that the identity functor is Morita compatible \cite[Lemma~4.21]{klqfunctor}.
In \cite[Theorem~4.22]{klqfunctor} we proved that the greatest lower bound of the family of all exact and Morita compatible coaction functors is itself exact and Morita compatible.
It is easy to check that the arguments can be used to prove the following more precise statement:

\begin{prop}\label{glb exact or morita}
Let $\TT$ be a nonempty family of coaction functors.
If every functor in $\TT$ is exact, then so is $\glb\TT$,
and if every functor in $\TT$ is Morita compatible then so is $\glb\TT$.
\end{prop}
In particular, there are both a smallest exact coaction functor
and a smallest Morita compatible coaction functor.

Every coaction functor $\tau$ determines a crossed-product functor 
$\cp^\tau$
on actions
by composing with the full-crossed-product functor
$(A,\alpha)\mapsto (A\rtimes_\alpha G,\what\alpha)$.
If $\tau$ is exact or Morita compatible then so is 
$\cp^\tau$,
and if $\tau\le \sigma$ then 
$\cp^\tau\le \cp^\sigma$.
However, if $\TT$ is a nonempty family of coaction functors,
and $\SS=\{\cp^\tau:\tau\in\TT\}$ is the associated family of crossed-product functors,
with respective greatest lower bounds $\glb\SS$ and $\glb\TT$,
then
\[
\cp^{\glb\TT}\le \glb\SS,
\]
but we do not know whether this is always an equality.
In particular (see \cite[Question~4.25]{klqfunctor},
we do not know whether the smallest exact and Morita compatible crossed-product functor is naturally isomorphic to the composition with full-crossed-product of the smallest exact and Morita compatible coaction functor.

A coaction functor $\tau$ is \emph{decreasing} if
there is a natural transformation $Q^\tau$
fitting into the embellishment
\[
\xymatrix{
&(A^m,\delta^m) \ar[dl]_{\psi_A} \ar[dr]^{q^\tau_A}
\\
(A,\delta) \ar[dr]_{\Lambda_A} \ar[rr]^{Q^\tau_A}
&&
(A^\tau,\delta^\tau) \ar[dl]^{\Lambda^\tau_A}
\\
&(A^n,\delta^n)
}
\]
of the diagram~\ref{co fn diag},
equivalently $\tau\le \id$ (the identity functor).
This property tends to simplify considerations of various properties of coaction functors,
mainly by replacing $q^\tau$ by $Q^\tau$.
For example, a decreasing coaction functor $\tau$ is Morita compatible if and only if whenever $(X,\zeta)$ is an $(A,\delta)-(B,\epsilon)$ imprimitivity-bimodule coaction,
there are an $A^\tau-B^\tau$ imprimitivity bimodule $X^\tau$
and a $Q^\tau_A-Q^\tau_B$ compatible imprimitivity-bimodule homomorphism
$Q^\tau_X\:X\to X^\tau$ \cite[Proposition~5.5]{klqfunctor}.

The most well-studied decreasing coaction functors are
determined by \emph{large ideals} of the Fourier-Stieltjes algebra $B(G)$,
i.e., nonzero $G$-invariant weak* closed ideals $E$ of $B(G)$.
The preannihilator $\pann E$ is an ideal of $C^*(G)$,
and, denoting the quotient map by
\[
q_E\:C^*(G)\to C^*_E(G):=C^*(G)/\pann E,
\]
for any coaction $(A,\delta)$
we let
\begin{align*}
A^E&=A/\ker \bigl((\id\otimes q_E)\circ\delta\bigr).
\end{align*}
Then $\delta$ descends to a coaction $\delta^E$ on the quotient $A^E$,
and the assignments $(A,\delta)\mapsto (A^E,\delta^E)$ determine a decreasing coaction functor $\tau_E$.
We write
\[
Q^E=Q^{\tau_E}\:A\to A^E.
\]

The maximalization functor is not decreasing, so is not of the form $\tau_E$ for any large ideal $E$.
Moreover, \cite[Example~3.16]{klqexact} gives an example of a decreasing coaction functor $\tau$ such that for every large ideal $E$ the restrictions of $\tau$ and $\tau_E$ to the subcategory of maximal coactions are not naturally isomorphic; in particular, $\tau$ is not itself of the form $\tau_E$.

We call the large ideal $E$ \emph{exact} if the coaction functor $\tau_E$ is exact.
It is quite frustrating that so far
we have 
few exact large ideals;
for arbitrary $G$
we only know of one exact large ideal, namely $B(G)$, and $\tau_{B(G)}$ is the identity functor.
If 
the group $G$ is exact,
then
it seems plausible
--- although we have not checked this --- that
$B_r(G)$ is also an exact large ideal,
and
would
obviously 
be
the smallest one.
The frustrating thing is that for arbitrary $G$ we do not know whether there is a smallest exact large ideal $E$.
On the other hand, for every large ideal $E$ the coaction functor $\tau_E$ is Morita compatible \cite[Proposition~6.10]{klqfunctor}.
We do not know whether the intersection of all exact large ideals is exact;
the best we can say for now is that the set of all exact large ideals is closed under finite intersections \cite[Theorem~3.2]{klqexact}.
In a similar vein, if $\FF$ is a collection of large ideals,
with intersection $F$,
we do not know whether $\tau_F$ is the greatest lower bound of $\{\tau_E:E\in \FF\}$.

A coaction functor $\tau$ has the \emph{ideal property} \cite[Definition~3.10]{klqexact}
if for every coaction $(A,\delta)$ and every strongly $\delta$-invariant ideal $I$ of $A$, letting $\iota\:I\hookrightarrow A$ denote the inclusion map, the induced map $\iota^\tau\:I^\tau\to A^\tau$ is injective.
For every large ideal $E$, the coaction $\tau_E$ has the ideal property \cite[Lemma~3.11]{klqexact}.
We do not know
an example of a decreasing coaction functor that is Morita compatible and does not have the ideal property (see \cite[Remark~3.12]{klqexact}).

\section{Maximalization of degenerate homomorphisms}\label{max gen sec}

Our main objects of study are coaction functors, which involve maximalization of coactions.
We will need to maximalize possibly degenerate homomorphisms.
Maximalization can be characterized by a universal property
(see \cite[Lemma~6.2]{fischer} for nondegenerate morphisms, and \cite{klqfunctor} for the classical case),
but this does not seem well-suited to handling possibly degenerate homomorphisms.
Instead, we rely upon the Fischer construction, which involves three steps:
first form the crossed product by the coaction,
then the crossed product by the dual action,
and finally \emph{destabilize}, which roughly means extract $A$ from $A\otimes\KK$.

Our strategy for maximalizing possibly degenerate homomorphisms
is to do it for each of the three steps in the Fischer construction,
then combine.
The steps are Lemmas~\ref{coact gen}, \ref{act gen}, and \ref{com gen},
which will be combined in \thmref{max gen}.

\begin{lem}\label{coact gen}
Let $(A,\delta)$ and $(B,\epsilon)$ be coactions,
and let $\phi\:A\to M(B)$ be a 
possibly degenerate
$\delta-\epsilon$ equivariant homomorphism.
Then there is a 
unique
homomorphism
\[
\phi\rtimes G\:A\rtimes_\delta G\to M(B\rtimes_\epsilon G)
\]
such that
\begin{equation}\label{phi rtimes delta}
\begin{split}
(\phi\rtimes G)\bigl(j_A(a)j_G^\delta(g)\bigr)&=j_B\circ\phi(a)j_G^\epsilon(g)
\\&\qquad\righttext{for all}a\in A,g\in C_c(G)\subset C^*(G).
\end{split}
\end{equation}
Moreover, $\phi\rtimes G$
is nondegenerate if $\phi$ is, and
is $\what\delta-\what\epsilon$ equivariant, and
if $\phi(A)\subset B$ then $(\phi\rtimes G)(A\rtimes_\delta G)\subset B\rtimes_\epsilon G$.
Finally,
given a third action $(C,\gamma)$
and a possibly degenerate $\epsilon-\gamma$ equivariant homomorphism
$\psi\:B\to M(C)$,
if 
either 
$\phi(A)\subset B$
or $\psi$ is nondegenerate
then
\[
(\psi\rtimes G)\circ(\phi\rtimes G)=(\psi\circ\phi)\rtimes G.
\]
\end{lem}

\begin{proof}
The first part is 
\cite[Lemma~A.46]{enchilada}, 
and the 
other statements
follow from direct calculation.
\end{proof}

For the next step, we need some ancillary lemmas.
Lemmas~\ref{idealizer}--\ref{abs} are completely routine --- we record them for convenient reference.
Lemmas~\ref{subalg coc}--\ref{perturb hom} are included to prepare for \lemref{act gen}.

\begin{lem}\label{idealizer}
Let $B$ be a $C^*$-algebra, and let $D$ and $E$ be $C^*$-subalgebras of $M(B)$.
Suppose that
\[
\clspn\{ED\}=D,
\]
so that also $\clspn\{DE\}=D$.
Then there is a unique homomorphism $\rho\:E\to M(D)$ such that
\[
\rho(m)d=md\righttext{for all}m\in E,d\in D,
\]
and moreover $\rho$ is nondegenerate.
\end{lem}

\begin{lem}\label{mult subalg}
Let $D$, $B$, and $F$ be $C^*$-algebras, with $D\subset M(B)$, and let $\nu\:F\to M(B)$ be a nondegenerate homomorphism.
Suppose that $\clspn\{\nu(F)D\}=D$.
Let
$E=\nu(F)$.
Let $\rho\:E\to M(D)$ be the homomorphism from \lemref{idealizer}.
Then $\tau:=\rho\circ\nu\:F\to M(D)$ is the unique nondegenerate homomorphism satisfying
\begin{equation}\label{nu tau}
\nu(f)d=\tau(f)d\righttext{for all}f\in F,d\in D.
\end{equation}
\end{lem}

\begin{lem}\label{abs}
Keep the notation from \lemref{mult subalg}, and let $C$ be another $C^*$-algebra.
Let $w\in M(F\otimes C)$.
Define
\begin{align*}
U&=(\nu\otimes\id)(w)\in M(E\otimes C)\subset M(B\otimes C)
\\
W&=(\tau\otimes\id)(w)\in M(D\otimes C).
\end{align*}
Then
\[
W=(\rho\otimes\id)(U),
\]
and
\[
Wm=Um\righttext{for all}m\in \wilde M(D\otimes C).
\]
\end{lem}

Let $D$, $B$, and $C$ be $C^*$-algebras, with $D\subset M(B)$.
Let $\sigma\:D\hookrightarrow M(B)$ be the inclusion map.
Then,
by \cite[Proposition~A.6]{enchilada},
$\sigma\otimes\id\:D\otimes C\hookrightarrow M(B\otimes C)$
extends canonically to an injective
homomorphism
\[
\bar{\sigma\otimes\id}\:\wilde M(D\otimes C)\to M(B\otimes C)
\]
that is continuous from the $C$-strict topology to the strict topology,
and we frequently identify $\wilde M(D\otimes C)$ with its image in $M(B\otimes C)$.

\begin{lem}\label{subalg coc}
Keep the notation from the Lemmas~\ref{idealizer}--\ref{abs},
and let $F=C_0(G)$, $C=C^*(G)$, and $w=w_G$.
Also let $\epsilon$ be a coaction of $G$ on $B$.
Suppose that $D$ is strongly $\epsilon$-invariant,
and let $\zeta=\epsilon|_D$.
Suppose that
$U:=(\nu\otimes\id)(w_G)$ is an $\epsilon$-cocycle, and
$W:=(\tau\otimes\id)(w_G)$ is a $\zeta$-cocycle.
Define
\begin{align*}
\wilde\epsilon&:=\ad U\circ\epsilon
\midtext{and}
\wilde\zeta:=\ad W\circ\zeta.
\end{align*}
Then
$D$ is also strongly $\wilde\epsilon$-invariant, and $\wilde\zeta=\wilde\epsilon|_D$.
\end{lem}

\begin{proof}
For $d\in D$, we have
\begin{align*}
\wilde\epsilon(d)
&=\ad U\circ\epsilon(d)
\\&=\ad U\circ\zeta(d)\righttext{(since $\zeta=\epsilon|_B$)}
\\&=\ad W\circ\zeta(d)\righttext{(by \lemref{abs})}
\\&=\wilde\zeta(d).
\end{align*}
Since $\wilde\zeta$ is a coaction of $G$ on $D$, we conclude that $D$ is strongly $\wilde\epsilon$-invariant.
\end{proof}

\begin{lem}\label{perturb hom}
Let $(A,\delta)$ and $(B,\epsilon)$ be coactions,
and let $\phi\:A\to M(B)$ be a possibly degenerate $\delta-\epsilon$ equivariant homomorphism.
Let $\mu\:C_0(G)\to M(A)$ and $\nu\:C_0(G)\to M(B)$ be nondegenerate homomorphisms,
and assume that
\[
\phi\bigl(a\mu(f)\bigr)=\phi(a)\nu(f)\righttext{for all}a\in A,f\in C_0(G).
\]
Define
\begin{align*}
V&=(\mu\otimes\id)(w_G)\in M(A\otimes C^*(G))
\\
U&=(\nu\otimes\id)(w_G)\in M(B\otimes C^*(G)).
\end{align*}
Suppose that $V$ is a $\delta$-cocycle and $U$ is an $\epsilon$-cocycle.
Define
\begin{align*}
\wilde\delta&=\ad V\circ\delta
\\
\wilde\epsilon&=\ad U\circ\epsilon.
\end{align*}
Then $\phi$ is also $\wilde\delta-\wilde\epsilon$ equivariant.
\end{lem}

\begin{proof}
Define
$D=\phi(A)$.
Then there is a unique coaction $\zeta$ of $G$ on $D$ such that the surjection $\phi\:A\to D$ is $\delta-\zeta$ equivariant.
It follows that $D$ is strongly $\epsilon$-invariant. Moreover, $\zeta=\epsilon|_D$,
since for all $d\in D$ we can choose $a\in A$ such that $d=\phi(a)$, and then
\begin{align*}
\zeta(d)
&=\zeta\circ\phi(d)
\\&=(\phi\otimes\id)\circ\delta(a)
\\&=\epsilon\circ\phi(a)
\righttext{(regarding $\wilde M(D\otimes C^*(G))\subset M(B\otimes C^*(G))$)}
\\&=\epsilon(d).
\end{align*}
The canonical extension $\bar\phi\:M(A)\to M(D)$ takes $\mu$ to a the unique nondegenerate homomorphism $\tau\:C_0(G)\to M(D)$ satisfying
\eqref{nu tau} with $F=C_0(G)$,
and the unitary
\[
W:=(\phi\otimes\id)(V)=(\tau\otimes\id)(w_G)
\]
is a $\zeta$-cocycle.
The hypotheses imply that 
$\nu(C_0(G))D=D$.
Thus we can apply \lemref{subalg coc}:
The right-front rectangle
(involving $D$ and $M(B)$)
of the diagram
\[
\xymatrix@C-20pt{
A \ar[rr]^-\phi \ar[dd]_{\wilde\delta} \ar[dr]_\phi
&&M(B) \ar[dd]^{\wilde\epsilon}
\\
&D \ar[dd]^(.2){\wilde\zeta} \ar@{^(->}[ur]
\\
\wilde M(A\otimes C^*(G)) \ar'[r]^-{\bar{\phi\otimes\id}}[rr] \ar[dr]_(.4){\phi\otimes\id}
&&M(B\otimes C^*(G))
\\
&\wilde M(D\otimes C^*(G)) \ar@{^(->}[ur]
}
\]
commutes,
and the left-front rectangle
(involving $A$ and $D$)
commutes by naturality of cocycles,
and therefore the rear rectangle
(involving $A$ and $M(B)$)
commutes, giving $\wilde\delta-\wilde\epsilon$ equivariance of $\phi$.
\end{proof}

We are now ready for the second step
of the Fischer construction
for possibly degenerate homomorphisms:

\begin{lem}\label{act gen}
Let $(A,\alpha,\mu)$ and $(B,\beta,\nu)$ be equivariant actions,
and let $\phi\:A\to M(B)$ be a possibly degenerate $\alpha-\beta$ equivariant homomorphism such that
\[
\phi\bigl(a\mu(f)\bigr)=\phi(a)\nu(f)\righttext{for all}a\in A,f\in C_0(G).
\]
Then there is a  unique \(possibly degenerate\) homomorphism
\[
\phi\rtimes G\:A\rtimes_\alpha G\to M(B\rtimes_\beta G)
\]
such that
\begin{equation}\label{phi rtimes alpha}
\begin{split}
(\phi\rtimes G)\bigl(i_A(a)i_G^\alpha(c)\bigr)&=i_B\circ\phi(a)i_G^\beta(c)
\\&\qquad\righttext{for all}a\in A,c\in C^*(G).
\end{split}
\end{equation}
Moreover, $\phi\rtimes G$
is nondegenerate if $\phi$ is, and
is $\wilde\alpha-\wilde\beta$ equivariant, and
\begin{equation}\label{phi rtimes G k}
\begin{split}
(\phi\rtimes G)\bigl(c(\mu\rtimes G)(k)\bigr)
&=(\phi\rtimes G)(c)(\nu\rtimes G)(k)
\\&\qquad\righttext{for all}c\in A\rtimes_\alpha G,k\in\KK.
\end{split}
\end{equation}
Also,
if $\phi(A)\subset B$ then $(\phi\rtimes G)(A\rtimes_\alpha G)\subset B\rtimes_\beta G$.
Finally,
given a third action $(C,\gamma)$
and a possibly degenerate $\beta-\gamma$ equivariant homomorphism
$\psi\:B\to M(C)$,
if 
either 
$\phi(A)\subset B$
or $\psi$ is nondegenerate
then
\[
(\psi\rtimes G)\circ(\phi\rtimes G)=(\psi\circ\phi)\rtimes G.
\]
\end{lem}

\begin{proof}
The first statement, up through \eqref{phi rtimes alpha}, is \cite[Remark~A.8 (4)]{enchilada},
the preservation of nondegeneracy is well-known,
and the 
last part, starting with ``Also'',
follows from direct calculation.
We must verify the $\wilde\alpha-\wilde\beta$ equivariance
and \eqref{phi rtimes G k}.
We first claim that for all
$c\in A\rtimes_\alpha G$,
$d\in C^*(G)$,
$a\in A$, and
$f\in C_0(G)$ we have
\begin{align}
(\phi\rtimes G)\bigl(c\,i_G^\alpha(d)\bigr)&=(\phi\rtimes G)(c)i_B^\beta(d)
\label{first}
\\
(\phi\rtimes G)\bigl(c\,i_A(a)\bigr)&=(\phi\rtimes G)(c)i_B\circ\phi(a)
\label{second}
\\
(\phi\rtimes G)\bigl(c\,i_A\circ\mu(f)\bigr)&=(\phi\rtimes G)(c)i_B\circ\nu(f).
\label{third}
\end{align}
\eqref{first}--\eqref{second}
follow by first replacing $c$ by appropriately chosen generators,
and to see 
\eqref{third}
we use nondegeneracy of $i_A$ and the Cohen factorization theorem to write
\[
c=c'\,i_A(b)\righttext{for}c'\in A\rtimes_\alpha G,b\in A,
\]
and then compute
\begin{align*}
(\phi\rtimes G)\bigl(c\,i_A\circ\mu(f)\bigr)
&=(\phi\rtimes G)\bigl(c'\,i_A(b)i_A\circ\mu(f)\bigr)
\\&=(\phi\rtimes G)\bigl(c'\,i_A(b\mu(f)\bigr)
\\&=(\phi\rtimes G)(c')i_B\circ\phi\bigl(b\mu(f)\bigr)
\\&=(\phi\rtimes G)(c')i_B\bigl(\phi(b)\nu(f)\bigr)
\\&=(\phi\rtimes G)(c')i_B\bigl(\phi(b)\bigr)i_B\bigl(\nu(f)\bigr)
\\&=(\phi\rtimes G)\bigl(c'i_A(b)\bigr)i_B\bigl(\nu(f)\bigr)
\\&=(\phi\rtimes G)(c)i_B\circ\nu(f).
\end{align*}
Combining \eqref{third} with the other hypotheses,
we can apply \lemref{perturb hom}
to conclude that $\phi\rtimes G$ is $\wilde\alpha-\wilde\beta$ equivariant.

For 
\eqref{phi rtimes G k},
it suffices to consider a generator
\[
k=i_{C_0(G)}(f)i_G^{\rt}(d)
\righttext{for}f\in C_0(G),d\in C^*(G),
\]
and then 
compute
\begin{align*}
(\phi\rtimes G)\bigl(c(\mu\rtimes G)(k)\bigr)
&=(\phi\rtimes G)\bigl(ci_A\circ\mu(f)i_G^\alpha(d)\bigr)
\\&=(\phi\rtimes G)\bigl(ci_A\circ\mu(f)\bigr)i_B^\beta(d)
\righttext{(by \eqref{first})}
\\&=(\phi\rtimes G)(c)i_B\circ\nu(f)i_B^\beta(d)
\righttext{(by \eqref{third})}
\\&=(\phi\rtimes G)(c)(\nu\rtimes G)(k).
\qedhere
\end{align*}
\end{proof}

Finally, we are ready for the third step
of the Fischer construction
for possibly degenerate homomorphisms:

\begin{lem}\label{com gen}
Let $(A,\delta,\iota)$ and $(B,\epsilon,\jmath)$ be $\KK$-coactions,
and let $\phi\:A\to M(B)$ be a possibly degenerate $\delta-\epsilon$ equivariant homomorphism such that
\[
\phi\bigl(a\iota(k)\bigr)=\phi(a)\jmath(k)\righttext{for all}a\in A,k\in\KK.
\]
Then there is a unique \(possibly degenerate\) homomorphism
\[
C(\phi)\:C(A,\iota)\to M(C(B,\jmath))
\]
making the diagram
\begin{equation}\label{C phi}
\xymatrix{
C(A,\iota)\otimes\KK \ar[r]^-{\theta_A}_-\simeq \ar[d]_{C(\phi)\otimes\id}
&A \ar[d]^\phi
\\
M(C(B,\jmath)\otimes\KK) \ar[r]_-{\theta_B}^-\simeq
&M(B)
}
\end{equation}
commute.
Moreover, $C(\phi)$
is nondegenerate if $\phi$ is, and
is $C(\delta)-C(\epsilon)$ equivariant.
Also, if $\phi(A)\subset B$ then $C(\phi)(C(A,\iota))\subset C(B,\jmath)$.
Finally,
given a third $\KK$-coaction
$(C,\zeta,\omega)$
and a possibly degenerate $\epsilon-\zeta$ equivariant homomorphism
$\psi\:B\to M(C)$ satisfying
$\psi(b\jmath(k))=\psi(b)\omega(k)$ for all $b\in B$ and $k\in\KK$,
if 
either 
$\phi(A)\subset B$
or $\psi$ is nondegenerate
then
\begin{equation}\label{C circ C}
C(\psi)\circ C(\phi)=C(\psi\circ\phi).
\end{equation}
\end{lem}

\begin{proof}
By \cite[Lemma~A.5]{dkq} $\phi$ extends uniquely to a homomorphism
\[
\bar\phi\:M_\KK(A)\to M(B)
\]
that is continuous from the $\KK$-strict topology to the strict topology.
Since $C(A,\iota)\subset M_\KK(A)$, we can define
\[
C(\phi)=\bar\phi|_{C(A,\iota)}.
\]
We will show that the diagram~\eqref{C phi} commutes,
and then the uniqueness will be obvious.
For $m\in C(A,\iota)$ and $k\in\KK$ we have
\begin{align*}
\theta_B\circ(C(\phi)\otimes\id)(m\otimes k)
&=\theta_B\bigl(\bar\phi(m)\otimes k\bigr)
\\&=\bar\phi(m)\jmath(k)
\\&\overset{*}=\phi\bigl(m\iota(k)\bigr)
\\&=\phi\circ\theta_A\bigl(m\otimes k\bigr),
\end{align*}
where the equality at $*$ follows from $\KK$-strict to strict continuity.
The preservation of nondegeneracy is proven in \cite[Theorem~4.4]{koqstable},
and follows from a routine approximate-identity argument.

For the equivariance,
let $f\in B(G)$, $m\in C(A,\iota)$, and $k\in\KK$.
Since $C(A,\iota)$ is a $B(G)$-submodule of $M(A)$,
we can compute as follows:
\begin{align*}
C(\phi)(f\cdot m)\jmath(k)
&=\bar\phi(f\cdot m)\jmath(k)
\righttext{(since $C(\phi)=\bar\phi|_{C(A,\iota)}$)}
\\&=\phi\bigl((f\cdot m)\iota(k)\bigr)
\righttext{(by \cite[Lemma~A.5]{dkq})}
\\&=\phi\Bigl(f\cdot \bigl(m\iota(k)\bigr)\Bigr)
\righttext{(since $\delta\circ\iota=\iota\otimes 1$)}
\\&=f\cdot \phi\bigl(m\iota(k)\bigr)
\righttext{(by \propref{equivariant})}
\\&=f\cdot \bigl(\bar\phi(m)\jmath(k)\bigr)
\\&=f\cdot \bigl(\bar\phi(m)\bigr)\jmath(k)
\\&=f\cdot \bigl(C(\phi)(m)\bigr)\jmath(k).
\end{align*}
Thus $C(\phi)(f\cdot m)=f\cdot C(\phi)(m)$ since $\jmath\:\KK\to M(B)$ is nondegenerate, and hence $\phi$ is equivariant by \propref{equivariant}.

Now suppose that $\phi(A)\subset B$.
Then for all $m\in C(A,\iota)$ and $k\in \KK$ we have
\begin{align*}
C(\phi)(m)\jmath(k)
&=\bar\phi(m)\jmath(k)
\\&=\phi\bigl(m\iota(k)\bigr)
\\&=\phi\bigl(\iota(k)m\bigr)
\\&=\jmath(k)\bar\phi(m)
\\&=\jmath(k)C(\phi)(m),
\end{align*}
which is an element of $B$ since $m\iota(k)\in A$.

The final statement, regarding composition, seems to not be recorded in the literature, so we give the proof
here.
First 
suppose that $\phi(A)\subset B$.
Then by \cite[Lemma~A.5]{dkq} the extension $\bar\phi$ maps $M_\KK(A)$ into $M_\KK(B)$ and is continuous for the $\KK$-strict topologies.
Also, $\bar\psi\:M_\KK(B)\to M(C)$ is continuous from the $\KK$-strict topology to the strict topology.
Let $\{a_i\}$ be a net in $A$ converging $\KK$-strictly to $m\in M_\KK(A)$.
Then $\phi(a_i)\to \bar\phi(m)$ $\KK$-strictly in $M_\KK(B)$,
and so
\[
\psi(\phi(a_i))\to \bar\psi(\bar\phi(m))\righttext{strictly in $M(C)$.}
\]
On the other hand, the composition
\[
\bar\psi\circ\bar\phi\:M_\KK(A)\to M(C)
\]
is continuous from the $\KK$-strict topology to the strict topology,
so
\[
\bar{\psi\circ\phi}(a_i)
\to \bar{\psi\circ\phi}(m).
\]
Since $\psi(\phi(a_i))=(\psi\circ\phi)(a_i)$ for all $i$,
we conclude that
\[
\bar\psi\circ\bar\phi(m)=\bar{\psi\circ\phi}(m).
\]
Since $C(\phi)$ and $C(\psi)$ are the restrictions to the relative commutants $C(A,\iota)$ and $C(B,\jmath)$, respectively, we get $C(\psi\circ\phi)=C(\psi)\circ C(\phi)$.

For the other case, where $\psi$ is nondegenerate,
we use the canonical extension of $\psi$ to $M(B)$ to compose, getting
a $\delta-\zeta$ equivariant homomorphism $\psi\circ\phi\:A\to M(C)$ such that
\[
(\psi\circ\phi)\bigl(a\iota(k)\bigr)=(\psi\circ\phi)(a)\omega(k)
\righttext{for all}a\in A,k\in\KK,
\]
so that $C(\psi\circ\phi)$ makes sense.
Since $C(\phi)$ is computed by restricting the canonical extension $\bar\phi\:M_\KK(A)\to M(B)$, and similarly for $C(\psi\circ\phi)$, and since we can compute the extension of $\psi$ on all of $M(B)$,
Equation~\eqref{C circ C} follows.
\end{proof}

We are now ready to maximalize possibly degenerate homomorphisms:

\begin{thm}\label{max gen}
Let $(A,\delta)$ and $(B,\epsilon)$ be coactions,
and let $\phi\:A\to M(B)$ be a possibly degenerate $\delta-\epsilon$ equivariant homomorphism.
Then there is a
unique \(possibly degenerate\) homomorphism $\phi^m\:A^m\to M(B^m)$
making
the diagram
\begin{equation}\label{phi m id}
\xymatrix{
&A\rtimes_\delta G\rtimes_{\what\delta} G \ar[dr]^{\Phi_A} \ar'[d][dd]^{\phi\rtimes G\rtimes G}
\\
A^m\otimes\KK \ar[ur]^(.4){\theta_{A\rtimes_\delta G\rtimes_{\what\delta} G}}_\simeq
\ar[rr]_(.3){\psi_A\otimes\id} \ar@{-->}[dd]_{\phi^m\otimes\id}
&&A\otimes\KK \ar[dd]^{\phi\otimes\id}
\\
&M(B\rtimes_\epsilon G\rtimes_{\what\epsilon} G) \ar[dr]^{\Phi_B}
\\
M(B^m\otimes\KK) \ar[ur]^\simeq_(.6){\theta_{B\rtimes_\epsilon G\rtimes_{\what\epsilon} G}}
\ar[rr]_-{\psi_B\otimes\id}
&&M(B\otimes\KK)
}
\end{equation}
commute,
where $\psi_A\:(A^m,\delta^m)\to (A,\delta)$ is the maximalization \(and similarly for $\psi_B$\).
Moreover,
$\phi^m$ is nondegenerate if $\phi$ is,
the diagram
\begin{equation}\label{phi m}
\xymatrix@C+20pt{
A^m \ar[r]^-{\phi^m} \ar[d]_{\psi_A}
&M(B^m) \ar[d]^{\psi_B}
\\
A \ar[r]_-\phi
&M(B)
}
\end{equation}
also commutes, and
$\phi^m$
is $\delta^m-\epsilon^m$ equivariant.
Also, if $\phi(A)\subset B$ then $\phi^m(A^m)\subset B^m$.
Finally,
given a third coaction
$(C,\zeta)$ and a possibly degenerate $\epsilon-\zeta$ equivariant homomorphism
$\pi\:B\to M(C)$,
if
either 
$\phi(A)\subset B$
or $\pi$ is nondegenerate
then
\[
(\pi\circ\phi)^m=\pi^m\circ\phi^m.
\]
\end{thm}

\begin{proof}
The right-rear rectangle
in the diagram~\eqref{phi m id}
(involving $A\rtimes G\rtimes G$ and $A\otimes\KK$)
commutes by direct computation.

Now,
$(A\rtimes_\delta G,\what\delta,j_G^\delta)$
and
$(B\rtimes_\epsilon G,\what\epsilon,j_G^\epsilon)$
are equivariant actions.
By \lemref{coact gen} 
the homomorphism
\[
\phi\rtimes G\:A\rtimes_\delta G\to M(B\rtimes_\epsilon G)
\]
is $\what\delta-\what\epsilon$ equivariant
and satisfies
\[
(\phi\times G)\bigl(cj_G^\delta(f)\bigr)=(\phi\rtimes G)(c)j_G^\epsilon(f)
\righttext{for all}c\in A\rtimes_\delta G,f\in C_0(G).
\]
Thus, by 
\lemref{act gen}
the homomorphism
\[
\phi\rtimes G\rtimes G\:A\rtimes_\delta G\rtimes_{\what\delta} G
\to M(B\rtimes_\epsilon G\rtimes_{\what\epsilon} G)
\]
is $\wilde\delta-\wilde\epsilon$ equivariant
and satisfies
\[
(\phi\rtimes G\rtimes G)\bigl(c(j_G^\delta\rtimes G)(k)\bigr)
=(\phi\rtimes G\rtimes G)(c)(j_G^\epsilon\rtimes G)(k)
\]
for all $c\in A\rtimes_\delta G\rtimes_{\what\delta} G$ and $k\in\KK$.
Furthermore,
$(A\rtimes_\delta G\rtimes_{\what\delta} G,\wilde\delta,j_G^\delta\rtimes G)$
and
$(B\rtimes_\epsilon G\rtimes_{\what\epsilon} G,\wilde\epsilon,j_G^\epsilon\rtimes G)$
are $\KK$-coactions.
Thus, by \lemref{com gen}
the homomorphism
\[
C(\phi\rtimes G\rtimes G)\:C(A\rtimes_\delta G\rtimes_{\what\delta} G,j_G^\delta\rtimes G)
\to M\bigl(C(B\rtimes_\epsilon G\rtimes_{\what\epsilon} G,j_G^\epsilon\rtimes G)\bigr)
\]
makes the diagram
\[
\xymatrix@C+30pt{
C(A\rtimes_\delta G\rtimes_{\what\delta} G,j_G^\delta\rtimes G)\otimes\KK
\ar[r]^-{\theta_{A\rtimes_\delta G\rtimes_{\what\delta} G}}_-\simeq
\ar[d]_{C(\phi\rtimes G\rtimes G)\otimes\id}
&A\rtimes_\delta G\rtimes_{\what\delta} G
\ar[d]^{\phi\rtimes G\rtimes G}
\\
M\bigl(C(B\rtimes_\epsilon G\rtimes_{\what\epsilon} G,j_G^\epsilon\rtimes G)\otimes\KK\bigr)
\ar[r]_-{\theta_{B\rtimes_\epsilon G\rtimes_{\what\epsilon} G}}^-\simeq
&M(B\rtimes_\epsilon G\rtimes_{\what\epsilon} G)
}
\]
commute.
Since
\[
A^m=C(A\rtimes_\delta G\rtimes_{\what\delta} G,i_{A\rtimes_\delta G}\circ j_G^\delta),
\]
by \lemref{com gen} we can define
\[
\phi^m=C(\phi\rtimes G\rtimes G),
\]
which is then the unique homomorphism making the left-rear rectangle
in the diagram~\eqref{phi m id}
(involving $A^m\otimes\KK$ and $A\rtimes G\rtimes G$)
commute.
The preservation of nondegeneracy follows immediately from the corresponding properties of the functors whose composition is $\phi\mapsto \phi^m$.
Then the front rectangle
(involving $A^m\otimes\KK$ and $A\otimes\KK$)
commutes,
and hence so does the diagram \eqref{phi m}.
Moreover, since $\delta^m=C(\delta)$ and $\epsilon^m=C(\epsilon)$, by \lemref{com gen} again we see that $\phi^m$ is $\delta^m-\epsilon^m$ equivariant.

For the final statement, involving composition,
suppose that we have $C$, $\zeta$, and 
$\pi$.
We consider the two cases separately:
first of all,
assume that $\phi(A)\subset B$.
Then from
\lemref{coact gen}
we conclude that
that
the equivariant actions
\begin{align*}
&(A\rtimes_\delta G,\what\delta,j_G^\delta)
\\
&(B\rtimes_\epsilon G,\what\epsilon,j_G^\epsilon)
\\
&(C\rtimes_\zeta G,\what\zeta,j_G^\zeta)
\end{align*}
and the homomorphisms
\begin{align*}
\phi\rtimes G\:&A\rtimes_\delta G\to B\rtimes_\epsilon G
\\
\pi\rtimes G\:&B\rtimes_\epsilon G\to M(C\rtimes_\zeta G)
\end{align*}
satisfy the hypotheses of 
\lemref{act gen}.
Thus, 
\lemref{act gen}
now tells us that
the $\KK$-coactions
\begin{align*}
&(A\rtimes_\delta G\rtimes_{\what\delta} G,\wilde\delta,j_G^\delta\rtimes G)
\\
&(B\rtimes_\epsilon G\rtimes_{\what\epsilon} G,\wilde\epsilon,j_G^\epsilon\rtimes G)
\\
&(C\rtimes_\zeta G\rtimes_{\what\zeta} G,\wilde\zeta,j_G^\zeta\rtimes G)
\end{align*}
and the homomorphisms
\begin{align*}
\phi\rtimes G\rtimes G\:&A\rtimes_\delta G\rtimes_{\what\delta} G\to B\rtimes_\epsilon G\rtimes_{\what\epsilon} G
\\
\pi\rtimes G\rtimes G\:&B\rtimes_\epsilon G\rtimes_{\what\epsilon} G\to M(C\rtimes_\zeta G\rtimes_{\what\zeta} G)
\end{align*}
satisfy the hypotheses of \lemref{com gen},
and hence, by construction of the maximalizations
$\delta^m,\epsilon^m,\zeta^m$ of $\delta,\epsilon,\zeta$,
we get
\[
\pi^m\circ\phi^m=(\pi\circ\phi)^m.
\]
On the other hand, if we assume that $\pi$ is nondegenerate instead of $\phi(A)\subset B$,
the argument proceeds similarly, except we keep tacitly using the canonical extension to multiplier algebras of any homomorphism constructed from $\pi$.
\end{proof}

\begin{rem}\label{max classical}
\thmref{max gen} gives a precise justification that the assignments
\begin{align*}
(A,\delta)&\mapsto (A^m,\delta^m)
\\
\phi&\mapsto \phi^m
\end{align*}
define a functor on the classical category of coactions.
\end{rem}

\section{Generalized homomorphisms}\label{gen hom}

\begin{defn}\label{gen def}
We say that a coaction functor $\tau$ is \emph{functorial for generalized homomorphisms} if whenever $(A,\delta)$ and $(B,\epsilon)$ are coactions and $\phi\:A\to M(B)$ is a possibly degenerate $\delta-\epsilon$ equivariant homomorphism
there is a (necessarily unique) possibly degenerate homomorphism $\phi^\tau$ making the following diagram commute:
\begin{equation}\label{gen def diag}
\xymatrix@C+20pt{
A^m \ar[r]^-{\phi^m} \ar[d]_{q_A^\tau}
&M(B^m) \ar[d]^{q_B^\tau}
\\
A^\tau \ar@{-->}[r]_-{\phi^\tau}
&M(B^\tau).
}
\end{equation}
Note that the existence of the homomorphism $\phi^m$ is guaranteed by \thmref{max gen}.
If $\phi^\tau$ is only presumed to exist when $\phi$ is nondegenerate, then we say that $\tau$ is \emph{functorial for nondegenerate homomorphisms}.
Note that if $\tau$ is functorial for generalized homomorphisms, it automatically sends nondegenerate homomorphisms to nondegenerate homomorphisms. This follows immediately from the corresponding property for the maximalization functor $A\mapsto A^m$.
\end{defn}

\begin{rem}
Let $\tau$ be a coaction functor,
and let $\cp^\tau$ be 
the associated crossed-product functor for actions,
given by full crossed product followed by $\tau$. If $\tau$ is functorial for generalized homomorphisms, then 
$\cp^\tau$
is also functorial for generalized homomorphisms in the sense
of \cite[Definition~3.1]{bew},
--- see \cite[paragraph following Definition~3.1]{bew}.
\end{rem}

Thus, a coaction functor $\tau$ is functorial for generalized homomorphisms if and only if for every possibly degenerate $\delta-\epsilon$ equivariant homomorphism $\phi\:A\to M(B)$ we have
\[
\ker q^\tau_A\subset \ker q^\tau_B\circ\phi^m,
\]
and similarly for nondegenerate functoriality.

\begin{ex}\label{max is gen}
The maximalization functor is functorial for generalized homomorphisms, by \thmref{max gen}.
Thus the identity functor $\id$ is functorial for generalized homomorphisms,
since we can take $q^{\id}_A=\psi_A$ and $\phi^{\id}=\phi$.
\end{ex}

\begin{rem}
Suppose that $\tau$ is functorial for generalized homomorphisms, and that $\phi\:A\to B$ is $\delta-\epsilon$ equivariant. Then the map $\phi^\tau$ vouchsafed by \defnref{gen def} agrees with the one that we get by the assumption that $\tau$ is a coaction functor.
In particular, if $\iota\:A\hookrightarrow M(A)$ is the canonical embedding then $\iota^\tau$ coincides with the canonical embedding $A^\tau\hookrightarrow M(A^\tau)$.
\end{rem}

\begin{lem}\label{compose gen}
Let $\tau$ be a coaction functor that is functorial for generalized homomorphisms,
let $(A,\delta)$, $(B,\epsilon)$, and $(C,\zeta)$ be coactions,
and let $\phi\:A\to M(B)$ and $\psi\:B\to M(C)$ be possibly degenerate equivariant homomorphisms.
If either $\phi(A)\subset B$
or $\psi$ is nondegenerate, 
then $(\psi\circ\phi)^\tau=\psi^\tau\circ\phi^\tau$.
\end{lem}

\begin{proof}
First assume that $\phi(A)\subset B$.
Then $\psi\circ\phi\:A\to M(C)$ is $\delta-\zeta$ equivariant.
Consider the diagram
\[
\xymatrix@C+30pt{
A^m \ar[rr]^-{\phi^m} \ar[dd]_{q^\tau_A} \ar[dr]_{(\psi\circ\phi)^m}
&&B^m \ar[dd]^{q^\tau_B} \ar[dl]^{\psi^m}
\\
&M(C^m) \ar[dd]^(.3){q^\tau_C}
\\
A^\tau \ar'[r]^-{\phi^\tau}[rr] \ar[dr]_{(\psi\circ\phi)^\tau}
&&B^\tau \ar[dl]^{\psi^\tau}
\\
&M(C^\tau).
}
\]
The top triangle commutes by \thmref{max gen}.
The rear, right-front, and left-front rectangles commute since $\tau$ is functorial for generalized homomorphisms.
Since the left vertical arrow $q^\tau_A$ is surjective,
it follows that the bottom triangle commutes, as desired.

On the other hand, assume that $\psi$ is nondegenerate.
Then again we have a $\delta-\zeta$ equivariant homomorphism
$\psi\circ\phi$ (extending $\psi$ canonically to $M(B)$),
the above diagram becomes
\[
\xymatrix@C+30pt{
A^m \ar[rr]^-{\phi^m} \ar[dd]_{q^\tau_A} \ar[dr]_{(\psi\circ\phi)^m}
&&M(B^m) \ar[dd]^{q^\tau_B} \ar[dl]^{\psi^m}
\\
&M(C^m) \ar[dd]^(.3){q^\tau_C}
\\
A^\tau \ar'[r]^-{\phi^\tau}[rr] \ar[dr]_{(\psi\circ\phi)^\tau}
&&M(B^\tau) \ar[dl]^{\psi^\tau}
\\
&M(C^\tau),
}
\]
and the argument proceeds as in the first part.
\end{proof}

Essentially the same techniques as in the above proof can be used to verify the following:

\begin{lem}\label{compose nd}
Let $\tau$ be a coaction functor that is functorial for nondegenerate homomorphisms,
let $(A,\delta)$, $(B,\epsilon)$, and $(C,\zeta)$ be coactions,
and let $\phi\:A\to M(B)$ and $\psi\:B\to M(C)$ be possibly degenerate equivariant homomorphisms.
If $\psi$ is nondegenerate,
and if
either $\phi(A)\subset B$
or $\phi$ is nondegenerate, 
then $(\psi\circ\phi)^\tau=\psi^\tau\circ\phi^\tau$.
In particular,
every coaction functor that is functorial for nondegenerate homomorphisms in the sense of \defnref{gen def} is also a functor on the nondegenerate category 
of coactions.
\end{lem}

As usual, things are simpler for decreasing coaction functors:

\begin{lem}\label{dec gen}
A decreasing coaction functor $\tau$ is functorial for generalized homomorphisms if and only if whenever $(A,\delta)$ and $(B,\epsilon)$ are coactions and $\phi\:A\to M(B)$ is a possibly degenerate $\delta-\epsilon$ equivariant homomorphism
there is a \(necessarily unique\) possibly degenerate homomorphism $\phi^\tau$ making the diagram
\begin{equation}\label{phi dec}
\xymatrix@C+20pt{
A \ar[r]^-\phi \ar[d]_{Q^\tau_A}
&M(B) \ar[d]^{Q^\tau_B}
\\
A^\tau \ar@{-->}[r]_-{\phi^\tau}
&M(B^\tau)
}
\end{equation}
commute.
If $\phi^\tau$ is only presumed to exist when $\phi$ is nondegenerate, then $\tau$ is functorial for nondegenerate homomorphisms.
\end{lem}

\begin{proof}
The above diagram fits into a bigger one:
\begin{equation}\label{big dec}
\xymatrix{
A^m \ar[rr]^-{\psi_A} \ar[dd]_{\phi^m} \ar[dr]_{q^\tau_A}
&&A \ar[dd]^\phi \ar[dl]^{Q^\tau_A}
\\
&A^\tau \ar@{-->}[dd]^(.3){\phi^\tau}
\\
M(B^m) \ar'[r]^-{\psi_B}[rr] \ar[dr]_{q^\tau_B}
&&M(B) \ar[dl]^{Q^\tau_B}
\\
&M(B^\tau).
}
\end{equation}
The top and bottom triangles commute since $\tau$ is a decreasing coaction functor.
The rear rectangle commutes
since the identity functor is functorial for generalized homomorphisms.
If 
there is a homomorphism $\phi^\tau$ making the left-front rectangle commute,
then the right-front rectangle also commutes since $\psi_A$ is surjective.
Conversely, if there is a homomorphism $\phi^\tau$ making the diagram~\eqref{phi dec} commute,
then the right-front rectangle in the diagram~\eqref{big dec} commutes,
and hence so does the left-front rectangle.
\end{proof}

Thus, a decreasing coaction functor $\tau$ is functorial for generalized homomorphisms if and only if for every possibly degenerate $\delta-\epsilon$ equivariant homomorphism $\phi\:A\to M(B)$ we have
\[
\ker Q^\tau_A\subset \ker Q^\tau_B\circ\phi.
\]

\begin{ex}\label{klq gen}
We apply \lemref{dec gen} to show that for every large ideal $E$ of $B(G)$, the coaction functor $\tau_E$ is functorial for generalized homomorphisms. Let $\phi\:A\to M(B)$ be a $\delta-\epsilon$ equivariant homomorphism,
and let
\[
a\in \ker Q^E_A
=\{b\in A:E\cdot a=\{0\}\}.
\]
Then for all $f\in E$ we have
\begin{align*}
f\cdot \phi(a)
&=\phi(f\cdot a)\righttext{(by equivariance)}
\\&=0,
\end{align*}
so $a\in \ker Q^E_B\circ\phi$.
In particular, 
the
identity functor and the normalization functor are functorial for generalized homomorphisms. For the identity functor this fact was already noted in \exref{max is gen}.
\end{ex}

\subsection*{The ideal property}

A coaction functor $\tau$ has the \emph{ideal property}
\cite[Definition~3.10]{klqexact}
if for every coaction $(A,\delta)$ and every strongly invariant ideal $I$ of $A$,
letting $\iota\:I\hookrightarrow A$ denote the inclusion map, the induced map
\[
\iota^\tau\:I^\tau\to A^\tau
\]
is injective.

\begin{ex}\label{id ideal}
The identity functor trivially has the ideal property.
\end{ex}

\begin{ex}
Every exact coaction functor has the ideal property,
and hence by \cite[Theorem~4.11]{klqfunctor} maximalization has the ideal property.
However, normalization has the ideal property, but is not exact unless $G$ is,
since by \cite[Proposition~4.24]{klqfunctor} the composition of an exact coaction functor with the full-cross-product functor is an exact crossed-product functor,
and the composition of normalization with the full-crossed-product functor is the reduced crossed product,
which is not an exact crossed-product functor unless $G$ is an exact group.
\end{ex}

\begin{rem}\label{gen iff id}
If a coaction functor $\tau$ has the ideal property,
then the associated crossed-product functor for actions 
has the ideal property in the sense of \cite[Definition~3.2]{bew},
since the full-crossed-product functor is exact \cite[Proposition~12]{gre:local}.
For crossed-product functors, 
\cite[Lemma~3.3]{bew} includes the fact that
functoriality for generalized homomorphisms and the ideal property are equivalent.
In the following proposition we show that part of this carries over to coaction functors.
However,
our naive attempts to adapt
the argument from \cite{bew} to show that
the ideal property implies functoriality for generalized homomorphisms
seem to
require that if $\phi\:A\to M(B)$ is a $\delta-\epsilon$ equivariant homomorphism then there is a 
strongly $\epsilon$-invariant
$C^*$-subalgebra $E$ of $M(B)$ 
containing both $B$ and 
$\phi(A)$,
which we have
unfortunately been unable to prove.
\end{rem}

\begin{prop}\label{gen id}
If a coaction functor $\tau$ is functorial for nondegenerate homomorphisms,
in particular if $\tau$ is functorial for generalized homomorphisms,
then $\tau$ has the ideal property.
\end{prop}

\begin{proof}
We adapt the proof from \cite{bew}:
let $(A,\delta)$ be a coaction
and let $I$ be a strongly $\delta$-invariant ideal of $A$.
Let $\phi\:I\hookrightarrow A$ be the inclusion map,
let $\psi\:A\to M(I)$ be the canonical map,
and let $\iota\:I\hookrightarrow M(I)$ be the canonical embedding.
Note that $\iota$ and $\psi$ are nondegenerate equivariant homomorphisms,
and $\phi$ is a classical equivariant homomorphism.
We have $\psi\circ\phi=\iota$,
so by \lemref{compose nd} we also have $\psi^\tau\circ\phi^\tau=\iota^\tau$.
Since $\iota^\tau$ is the canonical embedding $I^\tau\hookrightarrow M(I^\tau)$,
we conclude that $\phi^\tau$ is injective.
\end{proof}

\begin{rem}\label{klq id}
Combining \exref{klq gen} with \propref{gen id}, we recover \cite[Lemma~3.11]{klqexact}: for every large ideal $E$ of $B(G)$ the coaction functor $\tau_E$ has the ideal property.
In particular, the identity functor and the normalization functor have the ideal property (and for the identity functor we already noted this in \exref{id ideal}).
\end{rem}

\begin{ex}\label{pathological}
We adapt the techniques of \cite[Example~3.16]{klqexact} (which was in turn adapted from the techniques of \cite[Section~2.5 and Example~3.5]{bew})
to show that if $G$ is nonamenable then there is a decreasing coaction functor for $G$ that does not have the ideal property,
and hence is not exact,
and also by \propref{gen id} is not functorial for nondegenerate homomorphisms,
and a fortiori is not functorial for generalized homomorphisms.
Let
\[
\RR=\bigl\{\bigl(C[0,1)\otimes C^*(G),\id\otimes\delta_G\bigr)\bigr\},
\]
and for every coaction $(A,\delta)$
let $\RR_{(A,\delta)}$
be the collection of all triples $(B,\epsilon,\phi)$,
where either
$(B,\epsilon)\in\RR$
and $\phi\:A\to B$ is a $\delta-\epsilon$ equivariant homomorphism
or
$(B,\epsilon)=(A^n,\delta^n)$
and $\phi\:A\to A^n$ is the normalization map.
Then let
\[
\left(\bigoplus_{(B,\epsilon,\phi)\in \RR_{A,\delta}}(B,\epsilon),
\bigoplus_{(B,\epsilon,\phi)\in \RR_{A,\delta}}\epsilon\right)
\]
be the direct-sum coaction.
Define a nondegenerate
$\delta-\bigoplus_{(B,\epsilon,\phi)\in \RR_{A,\delta}}\epsilon$
equivariant homomorphism
\[
Q_A^\RR=\bigoplus_{(B,\epsilon,\phi)\in \RR_{A,\delta}}\phi\:
A\to M\left(\bigoplus_{(B,\epsilon,\phi)\in \RR_{A,\delta}} B\right),
\]
and let
$A^\RR=Q_A^\RR(A)$.
Then there is a unique coaction $\delta^\RR$ of $G$ on $A^\RR$ such that $Q_A^\RR$ is $\delta-\delta^\RR$ equivariant.
Moreover,
for every morphism $\phi\:(A,\delta)\to (B,\epsilon)$ in the classical category of coactions there is a unique homomorphism $\phi^\RR$ making the diagram
\[
\xymatrix{
(A,\delta) \ar[r]^-\phi \ar[d]_{Q_A^\RR}
&(B,\epsilon) \ar[d]^{Q_B^\RR}
\\
(A^\RR,\delta^\RR) \ar@{-->}[r]_-{\phi^\RR}
&(B^\RR,\epsilon^\RR)
}
\]
commute, giving a decreasing coaction functor $\tau^\RR$ with
$(A^{\tau_\RR},\delta^{\tau_\RR})=(A^\RR,\delta^\RR)$
and
$\phi^{\tau_\RR}=\phi^\RR$.

We will show that (assuming that $G$ is nonamenable) the coaction functor $\tau_\RR$ does not have the ideal property.
Consider the coaction
\[
(A,\delta)=\bigl(C[0,1]\otimes C^*(G),\id\otimes\delta_G\bigr).
\]
Then
\[
I:=C[0,1)\otimes C^*(G)
\]
is a strongly invariant ideal of $A$,
because $\delta$ restricts on $I$ to the coaction
\[
\delta_I:=\id_{C[0,1)}\otimes \delta_G.
\]

To see that $Q_I^\RR$ is faithful,
note that $\RR_{(I,\delta_I)}$ contains the triple $(I,\delta_I,\id)$.
On the other hand, to see that $Q_A^\RR$ is not faithful
on $I$,
note that, since
$I$ has no nonzero projections,
there is no nonzero homomorphism
from $C[0,1]$ to $I$,
and hence no nonzero homomorphism
from $A=C[0,1]\otimes C^*(G)$ to $I$,
and so the only morphism in $\RR_{(A,\delta)}$ is the normalization map
\[
\id\otimes \lambda\:C[0,1]\otimes C^*(G)\to C[0,1]\otimes C^*_r(G),
\]
which is not faithful
on $I$
because $G$ is nonamenable.
\end{ex}

\begin{prop}\label{glb gen}
Let $\TT$ be a nonempty family of coaction functors.
If every functor in $\TT$ is functorial for generalized homomorphisms, then so is $\glb\TT$.
\end{prop}

\begin{proof}
Let $\phi\:A\to M(B)$ be a $\delta-\epsilon$ equivariant homomorphism.
We must show
\[
\ker q^\sigma_A\subset \ker (q^\sigma_B\circ \phi^m),
\]
equivalently
\begin{equation}\label{show 1}
\phi^m(\ker q^\sigma_A)B^m\subset \ker q^\sigma_B.
\end{equation}
For each $\tau\in \TT$ we have
\[
\phi^m(\ker q^\tau_A)B^m\subset \ker q^\tau_B\subset \ker q^\sigma_B,
\]
so by linearity
\[
\phi^m\biggl(\spn_{\tau\in \TT}\ker q^\tau_A\biggr)B^m
=\spn_{\tau\in \TT}\phi^m(\ker q^\tau_A)B^m
\subset \ker q^\sigma_B,
\]
and hence by density and continuity
\[
\phi^m\biggl(\clspn_{\tau\in \TT}\ker q^\tau_A\biggr)B^m\subset \ker q^\sigma_B.
\]
By definition of greatest lower bound,
we have verified \eqref{show 1}.
\end{proof}

\begin{prop}\label{ideal gen}
Let $\TT$ be a nonempty family of coaction functors.
If every functor in $\TT$ has the ideal property, then so does $\glb\TT$.
\end{prop}

\begin{proof}
Let $(A,\delta)$ be a coaction, let $I$ be a strongly invariant ideal of $A$, and let $\iota\:I\hookrightarrow A$ denote the inclusion map. We must show that the induced map
\[
\iota^\sigma\:I^\sigma\to A^\sigma
\]
is injective,
equivalently
\begin{equation}\label{kernels}
\iota^m(\ker q^\sigma_I)=\iota^m(I^m)\cap \ker q^\sigma_A.
\end{equation}
We know that for every $\tau\in \TT$ the map
\[
\iota^\tau\:I^\tau\to A^\tau
\]
is injective.
The computation justifying \eqref{kernels} is the same as part of the proof of \cite[Theorem~4.22]{klqfunctor}:
\begin{align*}
\iota^m(\ker q^\sigma_I)
&=\iota^m\left(\clspn_{\tau\in \TT} \ker q^\tau_I\right)
\\&=\clspn_{\tau\in \TT} \iota^m(\ker q^\tau_I)
\\&=\clspn_{\tau\in \TT}\bigl(\iota^m(I^m)\cap \ker q^\tau_A\bigr)
\righttext{(since $\tau$ has the ideal property)}
\\&=\iota^m(I^m)\cap \clspn_{\tau\in \TT} \ker q^\tau_A
\\&\hspace{.5in}\text{(since all spaces involved are ideals in $C^*$-algebras)}
\\&=\iota^m(I^m)\cap \ker q^\sigma_A.
\qedhere
\end{align*}
\end{proof}

This might be an appropriate place to record a similar fact for decreasing coaction functors:

\begin{prop}\label{min dec}
The greatest lower bound of
any family of decreasing coaction functors
is itself decreasing.
\end{prop}

\begin{proof}
We first point out a routine fact:
if $\sigma$ and $\tau$ are coaction functors,
and if $\sigma\le \tau$ and $\tau$ is decreasing, then $\sigma$ is decreasing.
To see this,
let $(A,\delta)$ be a coaction.
Since $\sigma\le \tau$,
\[
\ker q^\tau_A\subset \ker q^\sigma_A.
\]
Since $\tau$ is decreasing,
\[
\ker \psi_A\subset \ker q^\tau_A.
\]
Thus $\ker \psi_A\subset \ker q^\sigma_A$, so $\sigma$ is decreasing.

Now let $\sigma$ be the greatest lower bound of $\TT$.
For every $\tau\in \TT$ we have $\sigma\le \tau$ and $\tau$ is decreasing,
so 
$\sigma$ is decreasing.
\end{proof}

\section{Correspondence property}\label{cor ppy sec}

Given $C^*$-algebras $A$ and $B$, recall that an \emph{$A-B$ correspondence} is a Hilbert $B$-module $X$ equipped with a homomorphism $\varphi_A\:A\to \LL(X)$,
inducing a left $A$-module structure via $ax=\varphi_A(a)x$.
We sometimes write $X={}_AX_B$ to emphasize $A$ and $B$.
If $A=B$ we call $X$ an \emph{$A$-correspondence}.

The closed span 
of the inner product,
written $\spn\{\<X,X\>_B\}$,
is an ideal of $B$, and
$X$ is \emph{full} if this ideal is dense.
By the Cohen-Hewitt factorization theorem,
the set $AX=\{ax:a\in A,x\in X\}$ is an $A-B$ subcorrespondence, and $X$ is \emph{nondegenerate} if $AX=X$.

If $\phi\:A\to M(B)$ is a homomorphism, the associated \emph{standard $A-B$ correspondence}, denoted by ${}_AB_B$, has left-module homomorphism $\varphi_A=\phi$.

If $X$ is an $A-B$ correspondence and $Y$ is a $C-D$ correspondence, a \emph{correspondence homomorphism} from $X$ to $Y$ is a triple $(\pi,\psi,\rho)$, where $\pi\:A\to C$ and $\rho\:B\to D$ are homomorphisms and $\psi\:X\to Y$ is a linear map such that $\psi(ax)=\pi(a)\psi(x)$, $\psi(xb)=\psi(x)\rho(b)$, and $\<\psi(x),\psi(y)\>_D=\rho(\<x,y\>_B)$
(and recall that the second property, involving $xb$, is automatic).
If $\pi$ and $\rho$ are understood we sometimes write $\psi$ for the correspondence homomorphism.
If $\pi$, $\psi$, and $\rho$ are all bijections then $\psi$ is a \emph{correspondence isomorphism},
and we write $X\simeq Y$.
If $A=C$, $B=D$, $\pi=\id_A$, and $\rho=\id_B$, we call $\psi$ an \emph{$A-B$ correspondence homomorphism},
and an \emph{$A-B$ correspondence isomorphism} is an $A-B$ correspondence homomorphism that is also a correspondence isomorphism.

An \emph{$A-B$ Hilbert bimodule} is an $A-B$ correspondence $X$ equipped with a left $A$-valued inner product ${}_A\<\cdot,\cdot\>$ that is compatible with the $B$-valued one. $X$ is \emph{left-full} if $\clspn\{{}_A\<X,X\>\}=A$; to avoid ambiguity we sometimes say $X$ is \emph{right-full} if $\clspn\{\<X,X\>_B\}=B$. If $X$ is both left and right-full it is an \emph{$A-B$ imprimitivity bimodule}.
We write $X^*$ for the \emph{reverse} $B-A$ Hilbert bimodule\footnote{Although the notation $\wilde X$ is perhaps more common, it would conflict with another usage of $\,\wilde{}$ we will need later.}.
The \emph{linking algebra} of an $A-B$ Hilbert bimodule $X$ is
$L(X)=\smtx{A&X\\X^*&B}$,
but we frequently just write $\smtx{A&X\\{*}&B}$ because the lower-left corner takes care of itself.
The linking algebra of the reverse bimodule is $L(X^*)=\smtx{B&X^*\\X&B}$.
The linking algebra of an $A-B$ correspondence $X$ is defined as the linking algebra of the associated (left-full) $\KK(X)-B$ Hilbert bimodule.

Recall from \cite[Definition~1.7]{enchilada} that if $X$ is an $A-B$ correspondence and $I$ is an ideal of $B$, then $XI$ is an $A-B$ subcorrespondence of $X$, and the ideal
\[
X\dashind I=X\dashind_B^A I:=\{a\in A:aX\subset XI\}
\]
of $A$ is said to be \emph{induced from $I$ via $X$}.
If $X\simeq Y$ as $A-B$ correspondences, then $X\dashind I=Y\dashind I$ for every ideal $I$ of $B$.

The quotient $X/XI$ becomes an $(A/X\dashind I)-(B/I)$ correspondence.

Let $J=\clspn\{\<X,X\>_B\}$.
Then $X$ is a nondegenerate right $J$-module and $J$ is an ideal of $B$, so
\[
XI=(XJ)I=X(JI)=X(JI).
\]
Thus $X\dashind I=X\dashind (JI)$.
Moreover, $X$ may also be regarded as an $A-J$ correspondence, and the quotient $X/XI$ may also be regarded as an $(A/X\dashind_J^A (JI))-(J/(JI))$ correspondence.

If $I$ and $J$ are ideals of $B$,
and we regard $J$ as a $J-B$ correspondence with the given algebraic operations,
then
\[
J\dashind_B^J I=\{a\in J:aJ\subset JI\}=JI.
\]
On the other hand,
regarding $B$ as a $J-B$ correspondence with the given algebraic operations,
then, since $BI=I$,
we nevertheless still get the same result:
\[
B\dashind_B^J I=\{a\in J:aB\subset I\}=J\cap I=JI.
\]

Given
a homomorphism $\phi\:A\to M(B)$
and an ideal $I$ of $B$,
and regard $B$ as the associated standard $A-B$ correspondence
(with left-module multiplication given by $a\cdot b=\phi(a)b$ for $a\in A$ and $b\in B$),
then
\[
B\dashind_B^A I=\{a\in A:\phi(a)B\subset I\}
\]
is sometimes denoted by $\phi^*(I)$.

Regarding $A$ as a standard $A-A$ correspondence, for every ideal $I$ of $A$ we have
$A\dashind_A^A I=I$.

If $X$ is an $A-B$ correspondence and $Y$ is a $B-C$ correspondence, we write $X\otimes_B Y$ for the balanced tensor product, which is an $A-C$ correspondence.
Letting $K=\KK(X)$, $X$ becomes a left-full $K-B$ Hilbert bimodule,
and
\[
{}_AX_B\simeq ({}_AK_K)\otimes_K({}_KY_B).
\]
Letting $J=\clspn\{\<X,X\>_B\}$, $X$ becomes a full $A-J$ correspondence, and
\[
{}_AX_B\simeq ({}_AX_J)\otimes_J({}_JB_B).
\]
By Rieffel's induction in stages theorem, if $X$ is an $A-B$ correspondence, $Y$ is a $B-C$ correspondence, and $I$ is an ideal of $C$, then
\[
(X\otimes_B Y)\dashind_C^A I=X\dashind_B^A Y\dashind_C^B I.
\]
If $X$ is an $A-B$ imprimitivity bimodule
then
\[
X^*\otimes_A X\simeq {}_BB_B,
\]
so if $I$ is an ideal of $B$, then
\[
X^*\dashind_A^B X\dashind_B^A I=I.
\]

Given actions $\alpha$ and $\beta$ of $G$ on $A$ and $B$, respectively, and an $\alpha-\beta$ compatible action $\gamma$ on $X$, we say $(X,\gamma)$ is an \emph{$(A,\alpha)-(B,\beta)$ correspondence action}.
The crossed product $X\rtimes_\gamma G$ is an $(A\rtimes_\alpha G)-(B\rtimes_\beta G)$ correspondence,
and we let $i_X\:X\to M(X\rtimes_\gamma G)$ denote the canonical $i_A-i_B$ compatible correspondence homomorphism.
Writing $\gamma^{(1)}$ for the induced action of $G$ on $\KK(X)$,
there is a canonical isomorphism
\[
\KK(X\rtimes_\gamma G)\simeq \KK(X)\rtimes_{\gamma^{(1)}} G,
\]
and, blurring the distinction between these two isomorphic algebras,
the left-module homomorphism of the crossed-product correspondence is given by
\[
\varphi_{A\rtimes_\alpha G}=\varphi_A\rtimes G
\:A\rtimes_\alpha G\to M(\KK(X)\rtimes_{\gamma^{(1)}} G).
\]
In particular, if $X$ is a left-full $A-B$ Hilbert bimodule,
then $X\rtimes_\gamma G$ is a left-full $(A\rtimes_\alpha G)-(B\rtimes_\beta G)$ bimodule,
and is moreover an imprimitivity bimodule if $X$ is.

Let $(X,\gamma)$ be an $(A,\alpha)-(B,\beta)$ correspondence action,
and let
$J=\clspn\{\<X,X\>_B\}$.
Then $J$ is a $\beta$-invariant ideal of $B$,
and we write $\eta$ for the action on $J$ gotten by restricting $\beta$.
As in \cite[Proposition~3.2]{enchilada}\footnote{The theory of \cite{enchilada} uses reduced crossed products, but
for the results of concern to us here
the same techniques handle the case of full crossed products.},
\[
\clspn\<X\rtimes_\gamma G,X\rtimes_\gamma G\>_{B\rtimes_\beta G}
=J\rtimes_\eta G,
\]
where the latter is identified with an ideal of $B\rtimes_\beta G$ in the canonical way.

If $(X,\gamma)$ is an $(A,\alpha)-(B,\beta)$ Hilbert bimodule  action
(so that also ${}_A\<\gamma_s(x),\gamma_s(y)\>=\alpha_s({}_A\<x,y\>)$),
there are a canonical $\beta-\alpha$ compatible action $\gamma^*$ on $X^*$
and a canonical isomorphism
\[
(X\rtimes_\gamma G)^*\simeq X^*\rtimes_{\gamma^*} G.
\]

Dually,
given coactions $\delta$ and $\epsilon$ of $G$ on $A$ and $B$, respectively, and a $\delta-\epsilon$ compatible coaction $\zeta$ on $X$, we say $(X,\zeta)$ is an \emph{$(A,\delta)-(B,\epsilon)$ correspondence coaction}.
The crossed product $X\rtimes_\zeta G$ is an $(A\rtimes_\delta G)-(B\rtimes_\epsilon G)$ correspondence,
and we let $j_X\:X\to M(X\rtimes_\zeta G)$ denote the canonical $j_A-j_B$ compatible correspondence homomorphism.
Writing $\zeta^{(1)}$ for the induced coaction of $G$ on $\KK(X)$,
there is a canonical isomorphism
\[
\KK(X\rtimes_\zeta G)\simeq \KK(X)\rtimes_{\zeta^{(1)}} G,
\]
and, blurring the distinction between these two isomorphic algebras,
the left-module homomorphism of the crossed-product correspondence is given by
\[
\varphi_{A\rtimes_\delta G}=\varphi_A\rtimes G
\:A\rtimes_\delta G\to M(\KK(X)\rtimes_{\zeta^{(1)}} G).
\]
In particular, if $X$ is a left-full $A-B$ Hilbert bimodule,
then $X\rtimes_\zeta G$ is a left-full $(A\rtimes_\delta G)-(B\rtimes_\epsilon G)$ bimodule,
and is moreover an imprimitivity bimodule if $X$ is.

Let $(X,\zeta)$ be an $(A,\delta)-(B,\epsilon)$ correspondence coaction, and let
$J=\clspn\{\<X,X\>_B\}$.
Then $J$ is a strongly $\epsilon$-invariant ideal of $B$
\cite[Lemma~2.32]{enchilada},
and we write $\eta$ for the coaction on $J$ gotten by restricting $\epsilon$.
As in \cite[Proposition~3.9]{enchilada},
\[
\clspn\<X\rtimes_\zeta G,X\rtimes_\zeta G\>_{B\rtimes_\epsilon G}
=J\rtimes_\eta G,
\]
where the latter is identified with an ideal of $B\rtimes_\epsilon G$ in the canonical way.

If $(X,\zeta)$ is an $(A,\delta)-(B,\epsilon)$ Hilbert-bimodule coaction
(so that also
${}_{M(A\otimes C^*(G))}\<\zeta(x),\zeta(y)\>=\delta({}_A\<x,y\>)$),
there are a canonical $\epsilon-\delta$ compatible coaction $\zeta^*$ on $X^*$
and a canonical isomorphism
\[
(X\rtimes_\zeta G)^*\simeq X^*\rtimes_{\zeta^*} G.
\]

If $(X,\gamma)$ is an $(A,\alpha)-(B,\beta)$ correspondence action,
the \emph{dual coaction} $\what\gamma$ on $X\rtimes_\gamma G$
is $\what\alpha-\what\beta$ compatible,
and dually
if $(X,\zeta)$ is an $(A,\delta)-(B,\epsilon)$ correspondence coaction,
the \emph{dual action} $\what\zeta$ on $X\rtimes_\zeta G$
is $\what\delta-\what\epsilon$ compatible.
Moreover, if $(X,\gamma)$ is an $(A,\alpha)-(B,\beta)$ Hilbert-bimodule action, the isomorphism
$(X\rtimes_\gamma G)^*\simeq X^*\rtimes_{\gamma^*} G$
is $\what\gamma^*-\what{\gamma^*}$ equivariant,
and dually
if $(X,\zeta)$ is an $(A,\delta)-(B,\epsilon)$ Hilbert bimodule coaction, the isomorphism
$(X\rtimes_\zeta G)^*\simeq X^*\rtimes_{\zeta^*} G$
is $\what\zeta^*-\what{\zeta^*}$ equivariant.

Given equivariant actions $(A,\alpha,\mu)$ and $(B,\beta,\nu)$,
and an $(A,\alpha)-(B,\beta)$ correspondence action $(X,\gamma)$,
by \cite[Lemma~6.1]{koqmaximal}
there is an
$\wilde\alpha-\wilde\beta$ compatible coaction\footnote{(recall from \secref{prelim} this notation involving tildes)}
$\wilde\gamma$ on $X\rtimes_\gamma G$ given by
\[
\wilde\gamma(y)=V_A\what\gamma(y)V_B^*.
\]
Moreover, if in fact $(X,\gamma)$ is a Hilbert bimodule action,
the isomorphism
$(X\rtimes_\gamma G)^*\simeq X^*\rtimes_{\gamma^*} G$
is $\wilde\gamma^*-\wilde{\gamma^*}$
equivariant\footnote{and here is where the notation ${}^*$ for the reverse bimodule is important}.

Given $\KK$-algebras $(A,\iota)$ and $(B,\jmath)$,
and an $A-B$ correspondence $X$,
\cite[Theorem~6.4 and its proof]{koqstable}
constructs a
$C(A,\iota)-C(B,\jmath)$ correspondence
$C(X,\iota,\jmath)$ given by
\[
C(X,\iota,\jmath)=\{x\in M(X):\iota(k)\cdot x=x\cdot \jmath(k)\in X\text{ for all }k\in\KK\}.
\]
Writing $\kappa\:\KK\to M(\KK(X))$ for the induced nondegenerate homomorphism,
there is a canonical isomorphism
\[
\KK\bigl(C(X,\iota,\jmath)\bigr)\simeq C\bigl(\KK(X),\kappa\bigr),
\]
and, blurring the distinction between these two isomorphic algebras,
the left-module homomorphism of the relative-commutant correspondence
is given by
\[
\varphi_{C(A,\iota)}=C(\varphi_A)
\:C(A,\iota)\to M\bigl(C(\KK(X),\kappa)).
\]
In particular, if $X$ is a left-full $A-B$ Hilbert bimodule,
then $C(X,\iota,\jmath)$ is a left-full $C(A,\iota)-C(B,\jmath)$ bimodule,
and is moreover an imprimitivity bimodule if $X$ is.

Given $\KK$-coactions $(A,\delta,\iota)$ and $(B,\epsilon,\jmath)$,
and an $(A,\delta)-(B,\epsilon)$ correspondence coaction $(X,\zeta)$,
by \cite[Lemma~6.3]{koqmaximal}
there is a
$C(\delta)-C(\epsilon)$ compatible coaction
$C(\zeta)$ on $C(X,\iota,\jmath)$
given by the restriction of
the canonical extension to $M(X)$ of
$\zeta$.
As before, let $J=\clspn\{\<X,X\>_B\}$,
and let $\eta=\epsilon|_J$ be the restricted coaction.
Letting $\rho\:B\to M(J)$ be the canonical homomorphism,
which is nondegenerate,
we can define a nondegenerate homomorphism
\[
\omega=\rho\circ\jmath\:\KK\to M(J),
\]
and $(J,\eta,\omega)$ is a $\KK$-coaction.
It is not hard to verify that
\[
\clspn\{\<C(X,\iota,\jmath),C(X,\iota,\jmath)\>_{C(B,\jmath)}\}
=C(J,\omega),
\]
which we identify with an ideal of $C(B,\jmath)$.

If $(A,\delta,\iota)$ and $(B,\epsilon,\jmath)$ are $\KK$-coactions
and
$X$ is an $(A,\delta)-(B,\epsilon)$ Hilbert bimodule coaction,
there is an isomorphism
\[
C(X,\iota,\jmath)^*\simeq C(X^*,\jmath,\iota)
\]
of $C(B,\jmath)-C(A,\iota)$ Hilbert bimodules,
and moreover this isomorphism is
$C(\zeta)^*-C(\zeta^*)$ equivariant.

Recall that the maximalization of a coaction $(A,\delta)$ is the coaction
\[
(A^m,\delta^m)=
\bigl(C(A\rtimes_\delta G\rtimes_{\what\delta} G,j_G^\delta\rtimes G),
C(\wilde\delta)\bigr),
\]
where
\[
\wilde\delta=\wilde{\what\delta}
=\ad V_{A\rtimes_\delta G}\circ\what{\what\delta}.
\]

\begin{defn}\label{max zeta}
Given coactions $(A,\delta)$ and $(B,\epsilon)$,
the \emph{maximalization} of
an $(A,\delta)-(B,\epsilon)$ correspondence coaction $(X,\zeta)$
is
the $(A^m,\delta^m)-(B^m,\epsilon^m)$
correspondence coaction
\[
(X^m,\zeta^m):=
\bigl(C(X\rtimes_\zeta G\rtimes_{\what\zeta} G,
j_G^\delta\rtimes G,j_G^\epsilon\rtimes G),C(\wilde\zeta)\bigr),
\]
where
\[
\wilde\zeta(y)
=\wilde{\what\zeta}(y)
=V_{A\rtimes_\delta G}\,\what{\what\zeta}(y)V_{B\rtimes_\epsilon G}
\]
for $y\in X^m$.
\end{defn}

There
is a canonical isomorphism
\begin{equation}\label{K iso}
\bigl(\KK(X^m),(\zeta^m)^{(1)}\bigr)\simeq \bigl(\KK(X)^m,(\zeta^{(1)})^m\bigr).
\end{equation}
Blurring the distinction between these two isomorphic algebras,
the left-module homomorphism of the $A^m-B^m$ correspondence $X^m$ is given by
\[
\varphi_{A^m}=\varphi_A^m
\:A^m\to M(\KK(X)^m)=M(\KK(X^m)).
\]
In particular, if $X$ is a left-full $A-B$ Hilbert bimodule,
then $X^m$ is a left-full $A^m-B^m$ Hilbert bimodule,
and is moreover an imprimitivity bimodule if $X$ is.

Letting $J=\clspn\{\<X,X\>_B\}$
with coaction $\eta=\epsilon|_J$ 
as before,
it follows from the above properties of the functors in the factorization of the Fischer construction that
\[
\clspn\{\<X^m,X^m\>_{B^m}\}=J^m,
\]
which we identify with an ideal of $B^m$.

If $(X,\zeta)$ is an $(A,\delta)-(B,\epsilon)$ Hilbert bimodule coaction,
then it follows from the properties of the steps in the Fischer construction that there is a canonical isomorphism
\[
(X^{m*},\zeta^{m*})\simeq (X^{*m},\zeta^{*m}).
\]

Let $\tau$ be a coaction functor, and let $(X,\zeta)$ be a Hilbert $(B,\epsilon)$-module coaction
(equivalently, a $(\C,\delta_\text{triv})-(B,\epsilon)$ correspondence coaction, where $\delta_\text{triv}$ is the trivial coaction on $\C$).
Then $X^m\ker q^\tau_B$ is a Hilbert $B^m$-submodule of $X^m$.
We define
\[
X^\tau=X^m/X^m\ker q^\tau_B,
\]
which is a Hilbert $B^\tau$-module,
and we further write
\[
q^\tau_X\:X^m\to X^\tau
\]
for the quotient map, which is a surjective homomorphism
of the Hilbert $B^m$-module $X^m$ onto the Hilbert $B^\tau$-module $X^\tau$.
It follows quickly from the definitions that there is a (necessarily unique) Hilbert-module homomorphism $\zeta^\tau$ making the diagram
\[
\xymatrix@C+30pt{
X^m \ar[r]^-{\zeta^m} \ar[d]_{q^\tau_X}
&\wilde M(X^m\otimes C^*(G)) \ar[d]^{q^\tau_X\otimes \id}
\\
X^\tau \ar@{-->}[r]_-{\zeta^\tau}
&\wilde M(X^\tau\otimes C^*(G))
}
\]
commute, and that $\zeta^\tau$ is moreover a coaction on the Hilbert $B^\tau$-module $X^\tau$.
Let
\[
(q^\tau_X)^{(1)}\:\KK(X^m)\to \KK(X^\tau)
\]
be the induced surjection, which is
equivariant for the induced coactions
$(\zeta^m)^{(1)}$ on $\KK(X^m)$ and
$(\zeta^\tau)^{(1)}$ on $\KK(X^\tau)$.

Recall from
\cite[Definition~4.16]{klqfunctor}
that
we call 
a coaction functor $\tau$ Morita compatible
if whenever $(X,\zeta)$ is an $(A,\delta)-(B,\epsilon)$ imprimitivity-bimodule coaction we have
\[
\ker q^\tau_A=X^m\dashind \ker q^\tau_B.
\]

\begin{rem}
\cite[Lemma~4.19]{klqfunctor} says that a coaction functor $\tau$ is Morita compatible if and only if 
for
every $(A,\delta)-(B,\epsilon)$ imprimitivity-bimodule coaction $(X,\zeta)$ 
the maximalization $X^m$
descends to an $A^\tau-B^\tau$ imprimitivity bimodule $X^\tau$.
Thus, if $\cp^\tau$ is the crossed-product functor given by $\tau$ composed with full-crossed-product, then Morita compatibility of $\tau$ implies that $\cp^\tau$ is \emph{strongly Morita compatible} in the sense of \cite[Definition~4.7]{bew}.
\end{rem}

\begin{ex}\label{klq mor}
The maximalization functor, and also the functors $\tau_E$ for large ideals $E$ of $B(G)$, are Morita compatible, by \cite[Lemma~4.15, Remark~4.18, and Proposition~6.10]{klqfunctor}.
\end{ex}

\begin{rem}
\cite[Proposition~5.5]{klqfunctor} can be equivalently stated as follows:
A decreasing coaction functor $\tau$ is Morita compatible if and only if
whenever $(X,\zeta)$ is an $(A,\delta)-(B,\epsilon)$ imprimitivity-bimodule coaction we have
\[
\ker Q^\tau_A=X\dashind_B^A \ker Q^\tau_B.
\]
\end{rem}

\begin{rem}
Let $(A,\delta)$ be a coaction, and let $I$ be a strongly $\delta$-invariant ideal of $A$.
The diagram
\begin{equation}\label{qIqA}
\xymatrix{
I^m \ar[r]^-{\iota^m} \ar[d]_{q^\tau_I}
&A^m \ar[d]^{q^\tau_A}
\\
I^\tau \ar[r]_-{\iota^\tau}
&A^\tau
}
\end{equation}
commutes because $\tau$ is a coaction functor.
The top arrow is always injective,
so we can identify $I^m$ with the ideal $\iota^m(I^m)$ of $A^m$.
Thus we always have
\[
\ker q^\tau_I
\subset \ker (q^\tau_A\circ \iota^m)
=I^m\cap \ker q^\tau_A,
\]
and since $\ker q^\tau_I\subset I^m$ we have $\ker q^\tau_I\subset \ker q^\tau_A$.
The ideal property for $\tau$ means that the bottom arrow is injective,
equivalently
\begin{equation}\label{intersect}
\ker q^\tau_I=I^m\cap \ker q^\tau_A,
\end{equation}
in which case the quotient map $q^\tau_I$ may be regarded as the restriction of $q^\tau_A$ to the ideal $I^m$.
\end{rem}

\begin{lem}\label{id ppy bimod}
Let $\tau$ be a coaction functor that has the ideal property.
Then $\tau$ is Morita compatible if and only if
for every left-full $(A,\delta)-(B,\epsilon)$ Hilbert-bimodule coaction $(X,\zeta)$ we have
\begin{equation}\label{left full}
\ker q^\tau_A=X^m\dashind_{B^m}^{A^m} \ker q^\tau_B.
\end{equation}
\end{lem}

\begin{proof}
The condition involving \eqref{left full} of course implies Morita compatibility, so suppose that $\tau$ is Morita compatible and $(X,\zeta)$ is a left-full $(A,\delta)-(B,\epsilon)$ Hilbert-bimodule coaction.

As before, let $J=\clspn\{\<X,X\>_B\}$ with the restricted coaction $\eta=\epsilon|_J$.
Then $(X,\zeta)$ is an $(A,\delta)-(J,\eta)$ imprimitivity-bimodule coaction, so
by Morita compatibility we have
\begin{equation}\label{AXJ}
\ker q^\tau_A=X^m\dashind_{J^m}^{A^m} \ker q^\tau_J.
\end{equation}

Identify $J^m$ with an ideal of $B^m$ in the usual way.
Regarding $B^m$ as a
standard $J^m-B^m$ correspondence,
we have
\begin{equation}\label{JB}
\ker q^\tau_J=J^m\cap \ker q^\tau_B=B^m\dashind_{B^m}^{J^m} \ker q^\tau_B.
\end{equation}
Thus by induction in stages we can combine \eqref{AXJ} and \eqref{JB} to conclude
that
\begin{align*}
\ker q^\tau_A
&=X^m\dashind_{B^m}^{A^m} \ker q^\tau_B.
\qedhere
\end{align*}
\end{proof}

\begin{defn}
We say that a coaction functor $\tau$ has the \emph{correspondence property} if for every $(A,\delta)-(B,\epsilon)$ correspondence coaction $(X,\zeta)$ we have
\[
\ker q^\tau_A\subset X^m\dashind_{B^m}^{A^m} \ker q^\tau_B.
\]
\end{defn}

Note that we have a commutative diagram
\[
\xymatrix{
A^m \ar[r]^-{\varphi_{A^m}} \ar[d]
&\LL(X^m) \ar[d]^{q^\tau_X}
\\
A^m/X^m\dashind \ker q^\tau_B \ar[r]
&\LL(X^\tau),
}
\]
with
\[
X^m\dashind \ker q^\tau_B=\ker (q^\tau_X\circ \varphi_{A^m}).
\]
The composition $q^\tau_X\circ \varphi_{A^m}$ gives $X^\tau$ a left $A^m$-module multiplication,
and
$\tau$ has the correspondence property if and only if
this left $A^m$-module multiplication on $X^\tau$
factors through a left $A^\tau$-module multiplication,
making $(X^\tau,\zeta^\tau)$ into a $(A^\tau,\delta^\tau)-(B^\tau,\epsilon^\tau)$ correspondence coaction.

\begin{ex}\label{max cor}
Trivially the maximalization functor has the correspondence property.
\end{ex}

\begin{thm}\label{cor mor gen}
A coaction functor $\tau$ has the correspondence property if and only if it is 
Morita compatible and
functorial for generalized homomorphisms.
\end{thm}

\begin{proof}
First assume that $\tau$ has the correspondence property.
For the Morita compatibility, let $(X,\zeta)$ be an $(A,\delta)-(B,\epsilon)$ imprimitivity bimodule coaction.
We must show that
\begin{equation}\label{ibm}
\ker q^\tau_A=X^m-\ind \ker q^\tau_B.
\end{equation}
By the correspondence property the left side is contained in the right side.
Since $(X^*,\zeta^*)$ is a $(B,\epsilon)-(A,\delta)$ imprimitivity bimodule coaction, we also have
\[
\ker q^\tau_B\subset X^{*m}\dashind \ker q^\tau_A.
\]
By induction in stages and the properties of reverse bimodules,
\begin{align*}
\ker q^\tau_A
&\subset X^m\dashind \ker q^\tau_B
\\&\subset X^m\dashind X^{*m}\dashind \ker q^\tau_A
\\&=\ker q^\tau_A,
\end{align*}
so we must have equality throughout,
and in particular \eqref{ibm} holds.

For the functoriality,
let $\phi\:A\to M(B)$ be a $\delta-\epsilon$ equivariant homomorphism.
Then $(B,\epsilon)$ is a standard $(A,\delta)-(B,\epsilon)$ correspondence coaction.
By assumption, we have
$\ker q^\tau_A\subset B^m\dashind \ker q^\tau_B$.
Since 
\[
B^m\dashind \ker q^\tau_B=
\{a\in A^m:\phi^m(a)B^m\subset \ker q^\tau_B\}=\ker(q^\tau_B\circ \phi^m),
\]
$\tau$ is functorial for generalized homomorphisms.

Conversely, assume that $\tau$ is Morita compatible and functorial for generalized homomorphisms.
Let $(X,\zeta)$ be an $(A,\delta)-(B,\epsilon)$ correspondence coaction.
We need to show that
\begin{equation}\label{AB}
\ker q^\tau_A\subset X^m\dashind_{B^m}^{A^m} \ker q^\tau_B.
\end{equation}

Let $K=\KK(X)$, with induced coaction $\mu$.
Let $\varphi_A\:A\to M(K)$ be the left-module homomorphism,
which is $\delta-\mu$ equivariant.
We use the associated $\delta^m-\mu^m$ equivariant homomorphism $\varphi_A^m\:A^m\to M(K^m)$
to regard $(K^m,\zeta^m)$ as a standard $(A^m,\delta^m)-(K^m,\mu^m)$ correspondence coaction.
By functoriality for generalized homomorphisms we have
\begin{equation}\label{AK}
\ker q^\tau_A\subset K^m\dashind_{K^m}^{A^m} \ker q^\tau_K.
\end{equation}

Note that $(X,\zeta)$ may be regarded as a left-full $(K,\mu)-(B,\epsilon)$ Hilbert-bimodule coaction.
Since $\tau$
is functorial for generalized homomorphisms,
by \propref{gen id} it
has the ideal property, so,
since $\tau$ is also assumed to be Morita compatible,
by \lemref{id ppy bimod} we have
\begin{equation}\label{KB}
\ker q^\tau_K=X^m\dashind_{B^m}^{K^m} \ker q^\tau_B.
\end{equation}
By induction in stages we can combine \eqref{AK} and \eqref{KB} to deduce \eqref{AB}.
\end{proof}

\begin{rem}
Although we do not need it
in the current paper,
it is natural to wonder whether
a coaction functor with the correspondence property 
will automatically be
functorial under composition of correspondences.
More precisely,
let $\tau$ be a coaction functor with the correspondence property,
and let $(X,\zeta)$ and $(Y,\eta)$ be
$(A,\delta)-(B,\epsilon)$ and $(B,\epsilon)-(C,\nu)$
correspondence coactions (respectively).
Then the balanced tensor product
$(X\otimes_B Y,\zeta\cotimes\eta)$
is a
$(A,\delta)-(C,\nu)$
correspondence coaction
(see \cite[Proposition~2.13]{enchilada}).
The assumption that $\tau$ has the correspondence property
implies that there are
$(A^\tau,\delta^\tau)-(B^\tau,\epsilon^\tau)$,
$(B^\tau,\epsilon^\tau)-(C^\tau,\nu^\tau)$,
and
$(A^\tau,\delta^\tau)-(C^\tau,\nu^\tau)$
correspondence coactions
$(X^\tau,\zeta^\tau)$,
$(Y^\tau,\eta^\tau)$,
and
$((X\otimes_B Y)^\tau,(\zeta\cotimes\eta)^\tau)$,
respectively.
The functoriality property we are 
wondering about
here is 
whether
there is a natural isomorphism
\[
\bigl((X\otimes_B Y)^\tau,(\zeta\cotimes\eta)^\tau\bigr)
\simeq
(X^\tau\otimes_{B^\tau} Y^\tau,\zeta^\tau\cotimes\eta^\tau)
\]
of
$(A^\tau,\delta^\tau)-(C^\tau,\nu^\tau)$
correspondence coactions.
It seems plausible that
this 
could
be checked via a 
tedious diagram chase,
or via linking algebras.
\end{rem}

\begin{ex}
Combining
\exref{klq gen},
\exref{klq mor}, and
\thmref{cor mor gen},
we see that $\tau_E$ has the correspondence property for every large ideal $E$ of $B(G)$.
\end{ex}

\begin{rem}
\thmref{cor mor gen} is similar to the equivalence (2)$\iff$(3) in \cite[Theorem~4.9]{bew}, except that, as we mentioned in \remref{gen iff id}, we have not been able to prove that for coaction functors the ideal property is equivalent to functoriality for generalized homomorphisms.
\end{rem}

\begin{rem}
\cite[Theorem~5.6]{bew} shows that every correspondence crossed-product functor produces $C^*$-algebras carrying a quotient of the dual coaction on the full crossed product.
This
reinforces our belief in the importance of 
studying
crossed-product functors
arising from coaction functors composed with the full cross product.
\end{rem}

\begin{cor}
Let $\TT$ be a nonempty family of coaction functors.
If every functor in $\TT$ has the correspondence property, then so does $\glb\TT$.
In particular, there is a smallest coaction functor with the correspondence property.
\end{cor}

Not surprisingly, the correspondence property is simpler for decreasing functors:

\begin{lem}\label{dec cor}
A decreasing coaction functor $\tau$ has the correspondence property if and only if for every $(A,\delta)-(B,\epsilon)$ correspondence coaction $(X,\zeta)$ we have
\[
\ker Q^\tau_A\subset X\dashind_B^A \ker Q^\tau_B.
\]
\end{lem}

\begin{proof}
We must show that the stated condition involving $Q^\tau_A$
holds if and only if
$\ker q^\tau_A\subset X^m\dashind_{B^m}^{A^m} \ker q^\tau_B$.
Let
\begin{align*}
&I=\ker \psi_A&&J=\ker \psi_B
\\
&K=\ker q^\tau_A&&L=\ker q^\tau_B.
\end{align*}
Then
$I\subset K\cap X^m\dashind J$,
$I\subset K$,
and
$J\subset L$,
and we can identify
$A$ with $A^m/I$,
$\ker Q^\tau_A$ with $K/I$,
$X$ with $X^m/X^mJ$,
$B$ with $B^m/J$
and
$\ker Q^\tau_B$ with $L/J$,
so the desired equivalence follows from the general \lemref{ind} below.
\end{proof}

In the proof of \lemref{dec cor} we appealed to the following elementary lemma, which is probably folklore.

\begin{lem}\label{ind}
Let $X$ be an $A-B$ correspondence,
let $I\subset K$ be ideals of $A$, and let $J\subset L$ be ideals of $B$.
Suppose that $I\subset X\dashind J$,
so that $X/XJ$ is an $(A/I)-(B/J)$ correspondence.
Then
$K\subset X\dashind L$
if and only if
$K/I\subset (X/XJ)\dashind L/J$.
\end{lem}

\begin{proof}
Let
\begin{align*}
\phi&\:A\to A/I
\\
\psi&\:X\to X/XJ
\\
\rho&\:B\to B/J
\end{align*}
be the quotient maps.
First assume that $K\subset X\dashind L$.
Then
\begin{align*}
(K/I)(X/XJ)
&=\phi(K)\psi(X)
\\&=\psi(KX)
\\&\subset \psi(XL)
\\&=\psi(X)\rho(L)
\\&=(X/XJ)(L/J),
\end{align*}
so $K/I\subset (X/XJ)\dashind L/J$.

Conversely, assume that
$K/I\subset (X/XJ)\dashind L/J$.
Then
\begin{align*}
KX
&\subset\psi\inv(\psi(KX))
\\&=\psi\inv\bigl(\phi(K)\psi(X)\bigr)
\\&\subset \psi\inv\bigl(\psi(X)\rho(L)\bigr)
\\&\overset{*}{=}\psi\inv(\psi(XL))
\\&=XL,
\end{align*}
where the equality at * holds since $\psi$ is a surjective homomorphism of correspondences and $XL$ is a closed subcorrespondence containing $\ker \psi=KJ$.
\end{proof}


\providecommand{\bysame}{\leavevmode\hbox to3em{\hrulefill}\thinspace}
\providecommand{\MR}{\relax\ifhmode\unskip\space\fi MR }
\providecommand{\MRhref}[2]{%
  \href{http://www.ams.org/mathscinet-getitem?mr=#1}{#2}
}
\providecommand{\href}[2]{#2}

\end{document}